\documentclass[final,onefignum,onetabnum]{siamart190516}
\usepackage{amssymb,amsmath,geometry,color}
\usepackage{graphicx}
\usepackage{epstopdf}
\usepackage{epsfig,multirow}
\usepackage{color,cases}
\usepackage{lipsum}
\usepackage{amsfonts}
\usepackage{graphicx}
\usepackage{epstopdf}
\usepackage{amsopn}
\ifpdf
  \DeclareGraphicsExtensions{.eps,.pdf,.png,.jpg}
\else
  \DeclareGraphicsExtensions{.eps}
\fi

\newsiamremark{remark}{Remark}
\newsiamremark{hypothesis}{Hypothesis}
\crefname{hypothesis}{Hypothesis}{Hypotheses}
\newsiamthm{claim}{Claim}

\newcommand{\abs}[1]{\left\vert#1\right\vert}

\newcommand{\norm}[1]{\left\Vert#1\right\Vert}

\def\sinc{\mathrm{sinc}}

\newtheorem{mydef}{Definition}[section]
\newtheorem{mytheo}{Theorem}[section]
\newtheorem{myprop}{Proposition}[section]

\newtheorem{mylemma}{Lemma}[section]

\headers{A third-order low-regularity trigonometric integrator}{B. Wang and Y.-L. Jiang}

\title{A third-order  trigonometric integrator with low regularity for the semilinear Klein-Gordon equation\thanks{Submitted to the editors DATE.
\funding{This work was supported by NSFC (12371403, 12271426) and Key Research and Development Projects of Shaanxi Province (2023-YBSF-399).}}}

\author{Bin Wang\thanks{School of Mathematics and Statistics, Xi'an Jiaotong University, 710049 Xi'an, China
  (wangbinmaths@xjtu.edu.cn).}
  \and Yaolin Jiang\thanks{School of Mathematics and Statistics, Xi'an Jiaotong University, 710049 Xi'an, China
  (yljiang@mail.xjtu.edu.cn).}
  }

\begin{document}

\maketitle

\begin{abstract}
In this paper, we propose and analyse a novel third-order low-regularity trigonometric integrator for the semilinear Klein-Gordon equation with non-smooth solution in the $d$-dimensional space, where $d=1,2,3$. The integrator is constructed based on the full use of Duhamel's formula and the  employment of a twisted function tailored for trigonometric integrals. Robust error analysis is conducted, demonstrating that the proposed scheme achieves third-order accuracy in the energy space under a weak regularity requirement in  $H^{1+\max(\mu,1)}(\mathbb{T}^d)\times H^{\max(\mu,1)}(\mathbb{T}^d)$ with $\mu> \frac{d}{2}$. A numerical experiment shows that the proposed third-order low-regularity integrator is much more accurate than some well-known exponential integrators of order three for approximating the Klein-Gordon equation with non-smooth solutions.

\end{abstract}



\section{Introduction}\label{sec:introduction}
This article concerns the numerical solution of the following
semilinear Klein--Gordon equation (SKGE):
\begin{equation}\label{klein-gordon}
\left\{
\begin{aligned}
&\partial_{tt}u(t,x)-\Delta u(t,x)+\rho u(t,x)=f\big(u(t,x)\big),\quad 0< t\leq T,\ x\in \mathbb{T}^d\subset \mathbb{R}^d,\\
&u(0,x)=u_0(x),\ \partial_{t}u(0,x)=v_0(x), \ x\in \mathbb{T}^d,
\end{aligned}
\right.
\end{equation}
under the periodic boundary condition,
where $u(t,x)$ represents the wave displacement at
time $t$ and position $x$, $d=1,2,3$ is the dimension of  $x$,  $\rho\geq 0$ is a given parameter  which is interpreted physically as the mass of the particle and   $f(u):\mathbb{R}\rightarrow \mathbb{R}$ is a given nonlinear function. Theoretically,  the SKGE is globally
well-posed for the initial data $(u_0(x),v_0(x)) \in H^{\nu}(\mathbb{T}^d)\times H^{\nu-1}(\mathbb{T}^d)$
with $\nu\geq 1/3$ in the one space dimension case  \cite{Bourgain} and with
$\nu\geq 1$ in the two or three dimensional space  \cite{Ginibre}.
There are many well-known examples of \eqref{klein-gordon}  coming from different choices of $f(u)$. For the very special case $f(u)=0$,
 the above Klein-Gordon equation is known as the relativistic version of the Schr\"{o}dinger
equation which correctly describes the spinless relativistic composite particles such as the pion and the Higgs
boson \cite{Davydov76}. In the case $f(u)=\sin(u)$,  the equation \eqref{klein-gordon}  corresponds to  the sine–Gordon equation. This case encompasses a wide array of physical applications, including the dissemination of dislocations in solid and liquid crystals, the transmission of ultra-short optical pulses within two-level systems, and the propagation of magnetic flux in Josephson junctions \cite{Barone71}.
If $f(u)=\lambda u^3$ with  a given dimensionless
parameter $\lambda$ (positive and negative for defocusing and focusing self-interaction, respectively), the equation covers the cubic wave equation under the choice of $\rho=0$ and this system   widely exists  in plasma physics \cite{Bellan14}.

 Classical time discretization methods have been extensively developed and researched in recent decades for solving SKGE, such as  splitting methods \cite{Bao22,Buchholz2018,Faou15}, trigonometric/exponential integrators \cite{Bao13,Buchholz2021,Gauckler15,Lubich99,WIW}, uniformly accurate methods \cite{Bao14,PI1,S2,Chartier15,WZ21}   and structure-preserving methods \cite{Cano06,Chen20,Cohen08-1,LI20,zhao18}. For classical methods and their analysis, strong regularity assumptions are unavoidable since usually every two temporal derivatives in the solution of
the  equation  with $f(u) \in  L^{\infty}$ can be converted to two spatial derivatives in the solution.
Therefore, in order to have $m$th-order   approximation of the solution  $(u,\partial_tu)$      in
the   space $H^{\nu}(\mathbb{T}^d)\times H^{\nu-1}(\mathbb{T}^d)$ with $\nu\geq d/2$,  the
initial data of \eqref{klein-gordon} is generally required  to be in the stronger space $H^{\nu+m-1}(\mathbb{T}^d)\times H^{\nu+m-2}(\mathbb{T}^d)$. When the accuracy of numerical methods is improved (i.e., $m$ becomes large),
higher requirement is needed for the regularity. However, such requirement may
not be satisfied since the initial data can be non-smooth in various real-world applications. For example, a high-intensity laser pulse is used to excite Langmuir waves in a plasma \cite{Borovskii1993}. The sharp leading edge of the laser pulse creates non-smooth initial conditions for the electric field, which can be modeled by the Klein-Gordon equation.

 To overcome this  barrier,  much attention has been paid to the
equations with non-smooth initial data recently. A pioneering work is \cite{A5}   for the KdV equation.  Another early work is \cite{O18}, where  a low-regularity
exponential-type integrator was designed for nonlinear  Schr\"{o}dinger equation and the proposed scheme could
have first-order convergence in  $H^{\nu}$ for
initial data in $H^{\nu+1}$ (the   regularity assumption is $H^{\nu+2}$ before the work \cite{O18}).
Since the  work \cite{A5,O18}, many different kinds of low-regularity (LR) integrators have been developed for various equations, including
Navier-Stokes equation \cite{LI22}, Dirac equation \cite{S21}, KdV equation \cite{R22,Zhao22},
  Schr\"{o}dinger equation \cite{O18,O21,O22,R21},  Boussinesq equation \cite{SU23} and Zakharov system \cite{SU24}.
To formulate the low-regularity
integrators  in a broader context,  many researches have  been done  in  \cite{A3,A1,A6,A4} and
a general class of low-regularity
integrators were introduced there for solving nonlinear  PDEs including the Sine-Gordon equation.

 For the system \eqref{klein-gordon} of Klein--Gordon equation,  some important integrators with  low-regularity property  have also been  formulated in recently years. First-order and second-order  methods were presented in \cite{R21} and the convergence of the second-order scheme was shown in the energy space $H^1(\mathbb{T}^d)\times L^2(\mathbb{T}^d)$ under the weaker regularity condition
 $(u_0(x),v_0(x))\in H^{1+\frac{d}{4}}(\mathbb{T}^d)\times H^{\frac{d}{4}}(\mathbb{T}^d)$.
For a cubic function $f(u)=u^3$ in one dimension $d=1$, a symmetric low-regularity  integrator was derived in \cite{Zhao23} which  possess second-order accuracy in $H^{\nu}(\mathbb{T}) \times H^{\nu-1}(\mathbb{T})$ with $\nu>1/2$ without loss of regularity of the solution.
 Recently, for the SKGE  \eqref{klein-gordon} with a nonlinear function $f(u)$,  second order convergence in  $H^1(\mathbb{T}^d)\times L^2(\mathbb{T}^d)$ was obtained for a new integrator in  \cite{LI23} under  the   regularity condition
 $(u_0(x),v_0(x))\in H^{1+\frac{d}{4}}(\mathbb{T}^d)\times H^{\frac{d}{4}}(\mathbb{T}^d)$ with $d=1,2,3.$
This regularity  requirement is  the same   as \cite{R21} but the proposed method has a much simple scheme.


As far as we know,  many existing specific low-regularity  integrators for the Klein--Gordon equation and/or other equations are devoted to first-order and second-order schemes.  It is an interesting question whether higher-order algorithm can achieve low regularity property
and can be formulated in a   brief manner.
The answer is positive  but the construction of specific  higher-order algorithms  is very challenging.  
By using the formalism of decorated trees  \cite{A3,A1,A6,A4}, various low-regularity
integrators can be derived successfully  including higher order schemes.
The objective of this article is to give  an alternative  and
brief  derivation of a third-order  low-regularity method.  We will construct and analyse a third-order  low-regularity  integrator (Definition \ref{def-se} given in Section \ref{sec:2})  for  the nonlinear Klein--Gordon equation \eqref{klein-gordon} with a general nonlinear function $f(u)$ and the dimension $d=1,2,3$.    The proposed
method   is stated in  the Duhamel’s formula and it has a simple scheme.
To this end, we first  embed the structure of SKGE  in the formulation and then apply   the  twisted function to the trigonometric integrals appeared in the Duhamel’s formulation. By carefully selecting
the tractable terms from the trigonometric integrals (see approximation to   I$^u$ and   I$^v$ given in Section \ref{sds-construct}), the spatial derivatives are
almost uniformly distributed to the product terms in the remainder and then the new integrator is obtained.
The proposed integrator will be shown to  have  third-order convergence in  $H^1(\mathbb{T}^d)\times L^2(\mathbb{T}^d)$ under the weaker regularity condition $(u_0(x),v_0(x))\in H^{1+\max(\mu,1)}(\mathbb{T}^d)\times H^{\max(\mu,1)}(\mathbb{T}^d)$ with $\mu> \frac{d}{2}$. Compared
with many existing LG methods, the new  integrator presented in this paper has a very simple scheme and it is very  convenient to  use
for the researchers even coming from different disciplines. Moreover, a fully-discrete  scheme is also proposed and studied in this article.
The error bounds  form spatial  and time discretisations
are simultaneously researched for SKGE \eqref{klein-gordon} with a non-smooth initial data.

The rest of this article is organized as follows. In Section \ref{sec:2}, we present the formulation of
the semi-discrete and fully-discrete methods. The error estimates of the proposed schemes are rigorously analysed in   Section \ref{sec:3}. Some  numerical  results are given in Section \ref{sec:4} to show the  performance of the new integrator in comparison with some existing methods in the literature. The  conclusions are drawn in  Section \ref{sec:5}.

\section{Construction of low-regularity  integrator}\label{sec:2}
In this section we will first introduce the notations and technical tools which are used frequently  in this paper.
Then  the new low-regularity integrator including semi-discrete and fully-discrete schemes will be presented in Subsection \ref{sds-construct} and the construction process will be given in
Subsection \ref{sds-construct}.
\subsection{Notations and technical tools}
 For simplicity,  we denote by $A\lesssim B$ the statement $A\leq CB$ for a generic constant $C>0$.
The constant $C$ may depend on $T$  and 	 $\norm{u_0}_{H^{1+\max(\mu,1)}}, \norm{v_0}_{H^{\max(\mu,1)}}$, and may be different at different
occurrences, but is always independent of the time and space discrete number of points and the time/space step sizes.

For the linear differential operator appeared in \eqref{klein-gordon}, we denote it by a concise notation \begin{equation}\label{notations A}
(\mathcal{A}w)(x):=-\Delta  w(x)+ \rho w(x)\end{equation} with $w(x)$ on $\mathbb{T}^d$.
In the construction of the integrator, we will  frequently use the following functions
\begin{equation}\label{notations w}\begin{aligned}&\sinc(t):=\frac{\sin(t)}{t},\ \quad   \qquad \alpha(t):=\left(
              \begin{array}{c}
                \cos(t\sqrt{\mathcal{A}})  \\
                t \sinc( t\sqrt{\mathcal{A}})  \\
              \end{array}
            \right),\\&\beta(t):=\left(
              \begin{array}{c}
                t \sinc( t\sqrt{\mathcal{A}})    \\
                \cos(t\sqrt{\mathcal{A}}) \\
              \end{array}
            \right),\  \gamma(t): =\left(
              \begin{array}{c}
                \sqrt{\mathcal{A}} \sin(t\sqrt{\mathcal{A}})    \\
                \cos(t\sqrt{\mathcal{A}}) \\
              \end{array}
            \right).
\end{aligned}\end{equation}

The Sobolev space of functions on the domain  $\mathbb{T}^d=[a,b]^d$ is considered in the whole paper,  and we shall refer it as  $H^{\nu}(\mathbb{T}^d)$ for any $\nu\geq 0$.
The norm on this Sobolev space    is denoted by
$$\norm{f}^2_{H^{\nu}}=(b-a)^{d}\sum_{\xi\in \mathbb{Z}^d} (1+\abs{\xi}^2)^{\nu}\abs{\widehat{f}(\xi)}^2,$$
where $\widehat{f}$ is the Fourier transform of   $f(x)$ which is defined by
$$\widehat{f}(\xi)=\frac{1}{(b-a)^d}\int_{\mathbb{T}^d} e^{-i x \cdot \xi}f(x) dx$$ with $\xi=(\xi^1,\xi^2,\ldots,\xi^d)\in \mathbb{Z}^d$
and $x=(x^1,x^2,\ldots,x^d)\in \mathbb{Z}^d$. Here $\abs{\cdot}$ is defined as
$\abs{\xi}=\sqrt{(\xi^1)^2 +(\xi^2)^2+\ldots+(\xi^d)^2}$ and $x \cdot \xi=x^1\xi^1+x^2\xi^2+\ldots+x^d\xi^d $.

For the norm on  Sobolev space,  some obvious properties are  frequently used in this article and we summarize them as follows.
 \begin{itemize}
   \item With the notation $J^{\nu}:=(1-\Delta)^{\frac{\nu}{2}}$ for $\nu\geq 0$, it is clear that
$$
\widehat{ J^{\nu}f}(\xi)=(1+\abs{\xi}^2)^{\frac{\nu}{2}}\widehat{f}(\xi),\ \ \norm{J^{\nu}f}_{L^2}=\norm{f}_{H^{\nu}}.
$$
   \item  For any function $\sigma: \mathbb{Z}^d\rightarrow \mathbb{C}^d$ such that $\abs{\sigma(\xi)}\leq C_{\sigma} (1+\abs{\xi}^2)^m$
with some constants $C_{\sigma}$ and $m$,  the operator  $\sigma( \sqrt{-\Delta}): H^{\nu}(\mathbb{T}^d)\rightarrow H^{\nu-m}(\mathbb{T}^d)$ has the following result  for $\nu\geq m$
$$\sigma(\sqrt{-\Delta})f(x)=\sum_{\xi\in \mathbb{Z}^d} \sigma(\xi)  \widehat{f}(\xi)  e^{i x \cdot \xi},\ \ \norm{\sigma(\sqrt{-\Delta})f}^2_{H^{\nu-m}}  \leq  C_{\sigma}  \norm{f}^2_{H^{\nu}}.$$

   \item  For the special  operator $ \sqrt{-\Delta}^{-1} $  which is defined by  $$\widehat{ \sqrt{-\Delta}^{-1}f}(\xi)=\left\{
\begin{split}
&
\frac{\widehat{ f}(\xi)}{\abs{\xi}},\quad
\textmd{when}\ \ \ \abs{\xi}\neq 0,\\
&0, \quad\ \ \ \ \
\textmd{when}\ \ \ \abs{\xi}= 0,
\end{split}\right.$$
it is straightforward to verify that
$\norm{\sqrt{-\Delta}^{-1}f}_{H^{\nu+1}}  \lesssim  \norm{f}_{H^{\nu}}.$
 \end{itemize}

For the operator $\mathcal{A}$  introduced in \eqref{notations A}, the following estimates are consequences of the properties stated above.
\begin{myprop}\label{lamma00}   If   $f\in H^{\nu}$ for any $\nu\geq 0$,  then the following results hold
$$\norm{\cos(t\sqrt{\mathcal{A}}) f}_{H^{\nu}}\leq \norm{f}_{H^{\nu}},\  \norm{\sin(t\sqrt{\mathcal{A}})  f}_{H^{\nu}}\leq \norm{f}_{H^{\nu}},$$
and $$ \norm{\sinc( t\sqrt{\mathcal{A}}) f}_{H^{\nu}}\leq \norm{f}_{H^{\nu}},\ \ \norm{\sqrt{\mathcal{A}}^{-1}f}_{H^{\nu+1}}\lesssim \norm{f}_{H^{\nu}}.$$

\end{myprop}
\begin{proof}The first three statements are obvious. For the last one, it is easy to check that
$\norm{\sqrt{\mathcal{A}}^{-1}f}_{H^{\nu+1}}=\norm{f}_{H^{\nu}}$ if $ \rho=1$,  $\norm{\sqrt{\mathcal{A}}^{-1}f}_{H^{\nu+1}}< \norm{f}_{H^{\nu}}$ if $ \rho>1$ and $\norm{\sqrt{\mathcal{A}}^{-1}f}_{H^{\nu+1}}< 2 \norm{f}_{H^{\nu}}$ if $ \rho<1$. Therefore, the proof is complete.
\end{proof}

The following version of the Kato–Ponce inequalities will also be needed in this paper, which was originally given in   \cite{Kato88}.

\begin{mylemma}\label{lamma0}(The Kato–Ponce inequalities)
In the regime $\nu> d/4$, we call upon the classical bilinear estimate
$$\norm{fg}_{H^{2 \nu}}\lesssim \norm{f}_{H^{2 \nu}}\norm{g}_{H^{2 \nu}},\ \ \norm{J^{-1}(Jfg)}_{H^{2 \nu}}\lesssim \norm{f}_{H^{2 \nu}}\norm{g}_{H^{2 \nu}},$$
whereas in the regime $0\leq\nu\leq d/4$, the following   bilinear estimates are true
 \begin{equation}\label{KPI}
\begin{aligned}
&\norm{fg}_{H^{2 \nu}}\lesssim \norm{f}_{H^{ \frac{d}{4}+\nu}}\norm{g}_{H^{\frac{d}{4}+\nu}}\ \ \ \textmd{for}\ \ 0\leq\nu< d/4,\\
&\norm{fg}_{H^{2 \nu}}\lesssim \norm{f}_{H^{ \frac{d}{2}+\epsilon}}\norm{g}_{H^{2 \nu}}\ \ \ \ \  \textmd{for}\ \ 0\leq\nu\leq d/4
\end{aligned}
\end{equation}
with any $\epsilon>0$.
 If $\nu>d/2$ and $f, g \in H^{\nu}$, we have the   inequality $$\norm{J^{\nu}(fg)}_{L^{2}}\lesssim \norm{f}_{H^{\nu }}\norm{g}_{H^{\nu}}.$$
 If $\delta \geq0, \nu> d/2$ and $f \in H^{\delta+\nu}$, $g \in H^{\delta}$, then
$$\norm{J^{\delta}(fg)}_{L^{2}}\lesssim \norm{f}_{H^{\delta+\nu }}\norm{g}_{H^{\delta}}.$$
\end{mylemma}

\subsection{Semi-discrete and fully-discrete  integrators}\label{sds}
We first directly present the scheme of semi-discrete integrator  and its deduction will be given in
Subsection \ref{sds-construct}.
\begin{mydef}\label{def-se} (\textbf{Semi-discrete  integrator})
 Let $t_n=n h $ with $n=0,1,\ldots, N$  be a uniform partition of the time interval $[0,T]$
with stepsize  $h=T/N$,  where $N$ is any given positive integer.
For solving the   Klein--Gordon
equation \eqref{klein-gordon}, we define the numerical solution $u_{n}=u_{n}(x)\approx u(t_n,x), v_{n}=v_{n}(x)\approx v(t_n,x):=\partial_tv(t_n,x)$ by the
following  semi-discrete  integrator
 \begin{equation}\label{method}
\begin{aligned}
u_{n+1}=&\cos(h \sqrt{\mathcal{A}}) u_n+h  \sinc( h \sqrt{\mathcal{A}})  v_n+h^2\Phi_1(h \sqrt{\mathcal{A}}) f( u_n)\\&+h^3  \Psi_1(h \sqrt{\mathcal{A}})f'( u_n)v_n +h^4  \Psi_2(h \sqrt{\mathcal{A}}) F_1(u_n,v_n),\\
v_{n+1}
   =&-h \mathcal{A} \sinc( h \sqrt{\mathcal{A}}) u_n+\cos(h \sqrt{\mathcal{A}}) v_n+
h  \Phi_2(h \sqrt{\mathcal{A}})  f( u_n) \\&+h^2\Phi_1(h \sqrt{\mathcal{A}})f'( u_n)v_n +h^3  \Psi_1(h \sqrt{\mathcal{A}}) \big(F_1(u_n,v_n)+f'(u_n) f(u_n) \big),
\end{aligned}
\end{equation}
where  $n=0,1,\ldots, T/h-1$, the initial value $u_0:=u_0(x), v_0:=v_0(x)$ is given in \eqref{klein-gordon}, the notation $F_1$ is
$$F_1(u_n,v_n)= f''(u_n)(v_n^2-(\nabla u_n)^2)+\rho  f(u_n)-\rho  f'(u_n)u_n,$$
 and  the coefficient functions of the integrator are defined by
 \begin{equation}\label{method-coefficient}
\begin{aligned}
 &\Phi_1(h \sqrt{\mathcal{A}})= \frac{\sinc( h \sqrt{\mathcal{A}})}{2},\ \ \  \ \ \Psi_2(h \sqrt{\mathcal{A}})= \frac{ 1-\cos( h \sqrt{\mathcal{A}})-\frac{h}{2} \sqrt{\mathcal{A}}\sin(h \sqrt{\mathcal{A}})}{(h \sqrt{\mathcal{A}})^4},\\
  &\Psi_1(h \sqrt{\mathcal{A}})= \frac{ \sinc( h \sqrt{\mathcal{A}})-\cos(h \sqrt{\mathcal{A}})}{2(h \sqrt{\mathcal{A}})^2},   \ \Phi_2(h \sqrt{\mathcal{A}})= \frac{ \cos(h \sqrt{\mathcal{A}}) + \sinc( h \sqrt{\mathcal{A}})}{2}.
\end{aligned}
\end{equation}
Here
the prime on $f$ indicates the   derivative  of $f(u)$ w.r.t. $u$.
\end{mydef}

\begin{remark}
From \eqref{method}, it can be observed clearly that this method is a kind of trigonometric integrators (\cite{hairer2006,Lubich99}). The scheme \eqref{method} is not complicated even in comparison with some first-order or second-order low-regularity integrators.  Therefore, it is convenient to be implemented in practical computations and has low computation cost. This aspect will be  demonstrated clearly by a numerical test given in Section \ref{sec:4}.
\end{remark}

Then we concern  the spatial discretization of  \eqref{method} which can be handled by using trigonometric interpolation \cite{Shen,LI23}.
To make the presentation be  more concise,  it is assumed that  $\mathbb{T}^d:=[0,1]^d$. Then  any function $W\in H^1_0(\mathbb{T}^d)\times L^2(\mathbb{T}^d)$
 can be expanded into the Fourier sine
series, i.e., $$W=\sum_{n_1,n_2,\ldots,n_d=1}^{\infty}W_{n_1,n_2,\ldots,n_d}\sin(n_1\pi x_1)\sin(n_2\pi x_2)\cdots\sin(n_d\pi x_d).$$ Choose a positive integer $N_x>0$ and denote by $I_{N_x}$and $\Pi_{N_x}$ the trigonometric interpolation and $L^2$-orthogonal projection operators onto $S_{N_x}$,
respectively, where the set $S_{N_x}$ is
  $$S_{N_x}=\{ \sum_{n_1,n_2,\ldots,n_d=1}^{N_x}W_{n_1,n_2,\ldots,n_d}\sin(n_1\pi x_1)\cdots\sin(n_d\pi x_d):\ W_{n_1,n_2,\ldots,n_d}\in \mathbb{R}^2\}.$$
With these notations and based on the semi-discrete scheme given in Definition \ref{def-se}, the    fully discrete low-regularity integrator   is defined as
follows.

\begin{mydef}\label{def-ful} (\textbf{Fully discrete  integrator})
For the semi-discrete integrator \eqref{method},  the fully discrete low-regularity integrator
is given by
 \begin{equation}\label{method-f}
\begin{aligned}
U_{n+1}=&\cos(h \sqrt{\mathcal{A}}) U_n+h  \sinc( h \sqrt{\mathcal{A}})  V_n+h^2\Phi_1(h \sqrt{\mathcal{A}})I_{N_x} f( U_n)\\&+h^3  \Psi_1(h \sqrt{\mathcal{A}})I_{N_x} \big(f'( U_n)V_n\big)+h^4  \Psi_2(h \sqrt{\mathcal{A}}) I_{N_x}F_1(U_n,V_n),\\
V_{n+1}
   =&-h \mathcal{A} \sinc( h \sqrt{\mathcal{A}}) U_n+\cos(h \sqrt{\mathcal{A}}) V_n\\&+
h  \Phi_2(h \sqrt{\mathcal{A}}) I_{N_x} f( U_n) +h^2\Phi_1(h \sqrt{\mathcal{A}})I_{N_x}\big(f'(U_n)V_n\big)  \\&+h^3  \Psi_1(h \sqrt{\mathcal{A}}) I_{N_x}\big(F_1(U_n,V_n)+ f'(U_n)  f(U_n)\big),
\end{aligned}
\end{equation}
where  $0<h <1$ is the time stepsize, $n=0,1,\ldots, T/h-1$ and the initial values are chosen as $U_0=\Pi_{N_x} u(x)$ and  $V_0=\Pi_{N_x} v(x)$ for
  $x\in \{\frac{2n}{2N_x+1}: n=1,2,\ldots,N_x \}^d$. In practical computation,  the trigonometric interpolation $I_{N_x}$ can be implemented with Fast Fourier Transform (FFT).
\end{mydef}

\subsection{Construction of semi-discrete integrator}\label{sds-construct}
In this subsection, we present the construction of the semi-discrete numerical method \eqref{method} based on twisted functions and
Duhamel’s formula.  For readers’ convenience,    the $x$-dependence of the unknown
functions is omitted and some technical
estimates of remainders are deferred to   Section \ref{sec:3}, where a rigorous error analysis will be given.

 For   the nonlinear Klein--Gordon
equation \eqref{klein-gordon},  the Duhamel’s formula at $t=t_n+s$ with $s\in \mathbb{R}$ reads
\begin{equation}\label{Duhamel}
\begin{aligned}
u(t_n+s)=&\cos(s\sqrt{\mathcal{A}}) u(t_n)+s \sinc( s\sqrt{\mathcal{A}})  v(t_n)\\&+\int_0^{s } (s-\theta)\sinc( (s-\theta)\sqrt{\mathcal{A}}) f\big( u(t_n+\theta)\big)d\theta,\\
v(t_n+s )=&-s\mathcal{A} \sinc( s\sqrt{\mathcal{A}}) u(t_n)+\cos(s\sqrt{\mathcal{A}})  v(t_n)\\&+\int_0^{s }  \cos( (s-\theta)\sqrt{\mathcal{A}}) f\big( u(t_n+\theta)\big)d\theta,
\end{aligned}
\end{equation}
which plays a crucial role in the method's formulation. Using this formula, we insert the expression of $u(t_n+\theta)$  into  the trigonometric integrals on the right hand side of \eqref{Duhamel}. Then by carefully selecting
the tractable terms from the obtained formulae and dropping some parts which do not affect third-order accuracy and regularity, the semi-discrete scheme is formulated. The detailed procedure is presented blew.

Firstly, we get the expression of $u(t_n+s)$ by the first equation of \eqref{Duhamel} and then insert this into
$f\big(u(t_n+ s)\big)$. This gives
\begin{equation}\label{funs}
\begin{aligned}
f\big(u(t_n+ s)\big)=&f\Big(\alpha^{\intercal}(s)U(t_n)\Big)+f'\Big(\alpha^{\intercal}(s)U(t_n)\Big)
\Big(u(t_n+ s)-\alpha^{\intercal}(s)U(t_n)\Big)\\
&+R_{f''}(t_n, s)\Big(u(t_n+ s)-\alpha^{\intercal}(s)U(t_n)\Big)^2,
\end{aligned}
\end{equation}
where $U(t_n)=(
             u(t_n),
             v(t_n) )^{\intercal}$   and \begin{equation}\label{RF2}\begin{aligned}&R_{f''}
=\int_0^1 \int_0^1\theta f''\Big( (1-\xi)\alpha^{\intercal}(s)U(t_n)+
\xi(1-\theta) \alpha^{\intercal}(s)U(t_n)+\theta u(t_n+ s)\Big)d\xi d\theta.\end{aligned}\end{equation}
Using the expression of $u(t_n+s)$ again in the right hand side of \eqref{funs}   yields the following expression:
\begin{equation}\label{funs-new}
\begin{aligned}
&f\big(u(t_n+ s)\big)\\
=&f\Big(\alpha^{\intercal}(s)U(t_n)\Big)
+f'\Big(\alpha^{\intercal}(s)U(t_n)\Big)
\int_0^{s }  \frac{\sin( ( s-\theta)\sqrt{\mathcal{A}})}{\sqrt{\mathcal{A}}} f\big( u(t_n+\theta)\big)d\theta \\
&+R_{f''}(t_n, s) \Big(\int_0^{s }  \frac{\sin( ( s-\theta)\sqrt{\mathcal{A}})}{\sqrt{\mathcal{A}}} f\big( u(t_n+\theta)\big)d\theta\Big)^2.
\end{aligned}
\end{equation}

Now we obtained a   desired form of $f\big(u(t_n+ s)\big)$ and based on which, the two trigonometric integrals  on the right hand side of \eqref{Duhamel} can be reformulated as
 \begin{equation}\label{local errors 1-new}
\begin{aligned}
&\int_0^{h }  \frac{\sin( (h -s)\sqrt{\mathcal{A}}) }{\sqrt{\mathcal{A}}} f\big( u(t_n+s)\big)ds
=\underbrace{\int_0^{h  }  \frac{\sin( (h -s)\sqrt{\mathcal{A}}) }{\sqrt{\mathcal{A}}}f\Big(\alpha^{\intercal}(s)U(t_n)\Big)ds}_{=:\textmd{  I}^{u}}\\
&+\underbrace{\int_0^{h  } \frac{\sin( (h -s)\sqrt{\mathcal{A}}) }{\sqrt{\mathcal{A}}}f'\Big(\alpha^{\intercal}(s)U(t_n)\Big) \int_0^{s }  \frac{\sin( ( s-\theta)\sqrt{\mathcal{A}})}{\sqrt{\mathcal{A}}} f\big( u(t_n+\theta)\big)d\theta ds}_{=:\textmd{  II}^{u}}
\\
&+\underbrace{\int_0^{h  }  \frac{\sin( (h -s)\sqrt{\mathcal{A}}) }{\sqrt{\mathcal{A}}}R_{f''}(t_n, s) \Big(\int_0^{s }  \frac{\sin( ( s-\theta)\sqrt{\mathcal{A}})}{\sqrt{\mathcal{A}}} f\big( u(t_n+\theta)\big)d\theta \Big)^2 ds}_{=:\textmd{  III}^{u}},
\end{aligned}
\end{equation}
and
 \begin{equation}\label{local errorsv 1-new}
\begin{aligned}
&\int_0^{h } \cos( (h-\theta)\sqrt{\mathcal{A}}) f\big( u(t_n+\theta)\big)d\theta=\underbrace{\int_0^{h  } \cos( (h -s)\sqrt{\mathcal{A}}) f\Big(\alpha^{\intercal}(s)U(t_n)\Big)ds}_{=:\textmd{  I}^{v}}\\
&+\underbrace{\int_0^{h  }\cos( (h -s)\sqrt{\mathcal{A}})f'\Big(\alpha^{\intercal}(s)U(t_n)\Big) \int_0^{s }  \frac{\sin( ( s-\theta)\sqrt{\mathcal{A}})}{\sqrt{\mathcal{A}}} f\big( u(t_n+\theta)\big)d\theta ds}_{=:\textmd{  II}^{v}}
\\
&+\underbrace{\int_0^{h  } \cos( (h -s)\sqrt{\mathcal{A}}) R_{f''}(t_n, s) \Big(\int_0^{s }  \frac{\sin( ( s-\theta)\sqrt{\mathcal{A}})}{\sqrt{\mathcal{A}}} f\big( u(t_n+\theta)\big)d\theta \Big)^2 ds}_{=:\textmd{  III}^{v}}.
\end{aligned}
\end{equation}

The main objective is to find some computable third-order approximations of I$^u$\&I$^v,$   II$^u$\&II$^v,$ and   III$^u$\&III$^v$, which will be derived one by one below.

 $\bullet$ \textbf{Approximation to   I$^u$ and   I$^v$.}

 We first deal with the integral I$^u$. To this end, consider the
 twisted variable  which  was  widely  used in the development
of low-regularity time discretizations   \cite{LI22,LI23,O18,O21,O22,R21,S21,Zhao23}.
In this paper, we  introduce a new twisted  function $$F(t_n+s):=\beta(-s) f\big( \alpha^{\intercal}(s) U(t_n)\big).$$
After splitting the term $\frac{\sin( (h -s)\sqrt{\mathcal{A}}) }{\sqrt{\mathcal{A}}}$ in I$^u$   into $\alpha^{\intercal}(h ) \beta(-s)$,
I$^u$   can be expressed by the twisted  function $F(t_n+s):$
 \begin{equation*}
\begin{aligned}
  \textmd{  I}^{u}
  =&\int_0^{h  }  \alpha^{\intercal}(h ) \beta(-s) f\big( \alpha^{\intercal}(s) U(t_n)\big) ds=\int_0^{h  }  \alpha^{\intercal}(h ) F(t_n+s) ds.
\end{aligned}
\end{equation*}
Considering  the Newton–Leibniz
formula for the twisted  function $F$:
$$F(t_n+s)=F(t_n)+ \int_0^{s }F'(t_n+\zeta) d\zeta,$$
we get
 \begin{equation}\label{P1-new1}
\begin{aligned}
  \textmd{  I}^{u}=&\int_0^{h  }  \alpha^{\intercal}(h ) F(t_n) ds+\int_0^{h  }  \alpha^{\intercal}(h ) \int_0^{s }F'(t_n+\zeta) d\zeta ds\\
 =&\int_0^{h  }  \alpha^{\intercal}(h ) F(t_n) ds+\int_0^{h  }  \alpha^{\intercal}(h )
 F'(t_n+\zeta)(h -\zeta) d\zeta.
\end{aligned}
\end{equation}
According to the scheme of $\alpha(h )$ given in
\eqref{notations w},  direct calculation yields the following expression
$$
  \alpha^{\intercal}(h )=\alpha^{\intercal}(h -s) M(s)\  \textmd{with}\ M(s):=\left(
                                 \begin{array}{cc}
                                  \cos(s\sqrt{\mathcal{A}}) &  s \sinc( s\sqrt{\mathcal{A}}) \\
                                 - \sqrt{\mathcal{A}} \sin( s\sqrt{\mathcal{A}}) &  \cos(s\sqrt{\mathcal{A}})  \\
                                 \end{array}
                               \right)
$$
 for any $s\in\mathbb{R}.$
Then
\eqref{P1-new1} can be expressed as
 \begin{equation*}
\begin{aligned}
\textmd{  I}^{u}=&\int_0^{h  }  \alpha^{\intercal}(h ) F(t_n) ds+\int_0^{h  }  (h -s) \alpha^{\intercal}(h -2s)
 M(2s)F'(t_n+s) ds\\
 =&\int_0^{h  }  \alpha^{\intercal}(h ) F(t_n) ds+\int_0^{h  }  (h -s) \alpha^{\intercal}(h -2s)
 M(0)F'(t_n) ds\\
 &+\int_0^{h  }  (h -s) \alpha^{\intercal}(h -2s)\int_0^{s}
 \frac{d  \big(M(2\zeta)F'(t_n+\zeta)\big)}{d\zeta}  d\zeta ds.
\end{aligned}
\end{equation*}
For the term $F'(t_n+\zeta)$ appeared above, it can be computed as:
 \begin{equation}\label{DF}
\begin{aligned}
&F'(t_n+\zeta)\\=&
 \frac{d}{d\zeta} \big( \beta(-\zeta)\big) f\big( \alpha^{\intercal}(\zeta) U(t_n)\big)+
            \beta(-\zeta)\frac{d}{d\zeta}  f\big( \alpha^{\intercal}(\zeta) U(t_n)  \\
            =&  \left(
              \begin{array}{c}
                - \cos(\zeta\sqrt{\mathcal{A}}) f\big( \alpha^{\intercal}(\zeta) U(t_n)\big) -\zeta \sinc(\zeta\sqrt{\mathcal{A}})\frac{d}{d\zeta}  f\big( \alpha^{\intercal}(\zeta) U(t_n)\big)   \\
             -\sqrt{\mathcal{A}}\sin(\zeta\sqrt{\mathcal{A}})f\big( \alpha^{\intercal}(\zeta) U(t_n)\big)+ \cos(\zeta\sqrt{\mathcal{A}}) \frac{d}{d\zeta}  f\big( \alpha^{\intercal}(\zeta) U(t_n)\big) \\
              \end{array}
            \right)\\
        =&M(-\zeta)\Big(  - f\big( \alpha^{\intercal}(\zeta) U(t_n)\big),
                                   \frac{d}{d\zeta}  f\big( \alpha^{\intercal}(\zeta) U(t_n)\big)
                              \Big)^{\intercal}.
\end{aligned}
\end{equation}
This result further leads to
 \begin{equation*}
\begin{aligned}
 &  \frac{d  }{d\zeta} M(2\zeta)F'(t_n+\zeta) =\frac{d  }{d\zeta} M(\zeta)\left(
                                \begin{array}{c}
                                  - f\big( \alpha^{\intercal}(\zeta) U(t_n)\big) \\
                                   \frac{d}{d\zeta}  f\big( \alpha^{\intercal}(\zeta) U(t_n)\big)  \\
                                \end{array}
                              \right)
 \\=&\Big(\frac{d  }{d\zeta} M(\zeta)\Big)\left(
                                \begin{array}{c}
                                  - f\big( \alpha^{\intercal}(\zeta) U(t_n)\big) \\
                                   \frac{d}{d\zeta}  f\big( \alpha^{\intercal}(\zeta) U(t_n)\big)  \\
                                \end{array}
                              \right)+M(\zeta) \frac{d  }{d\zeta} \left(
                                \begin{array}{c}
                                  - f\big( \alpha^{\intercal}(\zeta) U(t_n)\big) \\
                                   \frac{d}{d\zeta}  f\big( \alpha^{\intercal}(\zeta) U(t_n)\big)  \\
                                \end{array}
                              \right)
                               \\=&\left(
                                 \begin{array}{cc}
                                 -\sqrt{\mathcal{A}} \sin(\zeta\sqrt{\mathcal{A}}) &  \cos( \zeta\sqrt{\mathcal{A}}) \\
                                 - \mathcal{A} \cos( \zeta\sqrt{\mathcal{A}}) &  -\sqrt{\mathcal{A}} \sin(\zeta\sqrt{\mathcal{A}}) \\
                                 \end{array}
                               \right)\left(
                                \begin{array}{c}
                                  - f\big( \alpha^{\intercal}(\zeta) U(t_n)\big) \\
                                   \frac{d}{d\zeta}  f\big( \alpha^{\intercal}(\zeta) U(t_n)\big)  \\
                                \end{array}
                              \right)\\&\ \ +M(\zeta)  \left(
                                \begin{array}{c}
                                  - \frac{d  }{d\zeta}f\big( \alpha^{\intercal}(\zeta) U(t_n)\big) \\
                                   \frac{d^2}{d\zeta^2}  f\big( \alpha^{\intercal}(\zeta) U(t_n)\big)  \\
                                \end{array}
                              \right)
\\
                              =&\left(
                                \begin{array}{c}
                                 \sqrt{\mathcal{A}} \sin(\zeta\sqrt{\mathcal{A}})f\big( \alpha^{\intercal}(\zeta) U(t_n)\big)
                                 +\zeta \sinc( \zeta\sqrt{\mathcal{A}}) \frac{d^2}{d\zeta^2}  f\big( \alpha^{\intercal}(\zeta) U(t_n)\big)   \\
                                  \mathcal{A} \cos(\zeta\sqrt{\mathcal{A}}) f\big( \alpha^{\intercal}(\zeta) U(t_n)\big) +\cos(\zeta\sqrt{\mathcal{A}}) \frac{d^2}{d\zeta^2}  f\big( \alpha^{\intercal}(\zeta) U(t_n)\big)   \\
                                \end{array}
                              \right)
                             \\=&\beta(\zeta)\Upsilon(t_n,\zeta),
\end{aligned}
\end{equation*}
with the notation \begin{equation}\label{UPS}
\begin{aligned}\Upsilon(t_n,\zeta):=&   \mathcal{A} f\big( \alpha^{\intercal}(\zeta) U(t_n)\big)
                                 +\frac{d^2}{d\zeta^2}  f\big( \alpha^{\intercal}(\zeta) U(t_n)\big) \\
                                 =&   \mathcal{A} f\big( \alpha^{\intercal}(\zeta) U(t_n)\big)
                                 +   f''\big( \alpha^{\intercal}(\zeta) U(t_n)\big)\big(\gamma^{\intercal}(-\zeta) U(t_n),\gamma^{\intercal}(-\zeta) U(t_n) \big)\\&- f'\big( \alpha^{\intercal}(\zeta) U(t_n)\big)\mathcal{A}\alpha^{\intercal}(\zeta) U(t_n)\\
                                     =&f''\big( \alpha^{\intercal}(\zeta) U(t_n)\big)\big(\gamma^{\intercal}(-\zeta) U(t_n),\gamma^{\intercal}(-\zeta) U(t_n) \big)\\&-  f''\big( \alpha^{\intercal}(\zeta) U(t_n)\big)\big(\nabla \alpha^{\intercal}(\zeta) U(t_n),\nabla \alpha^{\intercal}(\zeta) U(t_n) \big)\\
                                     &+\rho  f\big( \alpha^{\intercal}(\zeta) U(t_n)\big)-\rho  f'\big( \alpha^{\intercal}(\zeta) U(t_n)\big)\alpha^{\intercal}(\zeta) U(t_n). \end{aligned}
\end{equation}
Here in the last equation, we use the definition of  $\mathcal{A}$  \eqref{notations A} and split it into $-\Delta $ and  $\rho$.
It is noted that the result of $\Upsilon(t_n,\zeta)$ is the key point  for reducing the regularity since it changes $ \mathcal{A}$ into $\nabla $ in the expression. This good aspect comes from the careful treatment of the twisted  function $F$ and the splitting of $\alpha^{\intercal}(h )$.

With the above results and the splitting of $\Upsilon(t_n,\zeta)$, we  derive that
\begin{equation*}
\begin{aligned}
\textmd{  I}^{u}=&\int_0^{h  }  \alpha^{\intercal}(h )  ds F(t_n) +\int_0^{h  }  (h -s) \alpha^{\intercal}(h -2s)
 M(0)ds F'(t_n) \\
 &+\int_0^{h  }  (h -s) \alpha^{\intercal}(h -2s)\int_0^{s}
\beta(\zeta)\Upsilon(t_n,\zeta)  d\zeta ds\\
 =&h ^2\sinc(  h \sqrt{\mathcal{A}}) f\big( u(t_{n})\big)  + \left(
                                                                   \begin{array}{c}
                                                                      h ^2/2\sinc(  h \sqrt{\mathcal{A}})  \\
                                                                       \frac{ \sin (  h \sqrt{\mathcal{A}}) -  h \sqrt{\mathcal{A}}\cos(  h \sqrt{\mathcal{A}})}{2 \sqrt{\mathcal{A}}^3}\\
                                                                   \end{array}
                                                                 \right)^{\intercal}
 \left(
                                \begin{array}{c}
                                  - f\big( u(t_n)\big) \\
                                  f'\big( u(t_n)\big) v(t_n)  \\
                                \end{array}
                              \right) \\
 &+\int_0^{h  }  (h -s) \alpha^{\intercal}(h -2s)\int_0^{s}
\beta(\zeta)(\Upsilon(t_n,0)+\Upsilon(t_n,\zeta)-\Upsilon(t_n,0))  d\zeta ds\\
                              =&h ^2/2\sinc(  h \sqrt{\mathcal{A}}) f\big( u(t_{n})\big)+\frac{ \sin (  h \sqrt{\mathcal{A}}) -  h \sqrt{\mathcal{A}}\cos(  h \sqrt{\mathcal{A}})}{2 \sqrt{\mathcal{A}}^3} f'\big( u(t_n)\big) v(t_n) \\
 &+\frac{ 2-2\cos( h \sqrt{\mathcal{A}})-h \sqrt{\mathcal{A}}\sin(h \sqrt{\mathcal{A}})}{2   \mathcal{A}^2} \Upsilon(t_n,0)+R_1(t_n),\end{aligned}
\end{equation*}
where  the remainder $R_1$ is defined as
\begin{equation}\label{R1}R_1(t_n)=\int_0^{h  }  (h -s) \alpha^{\intercal}(h -2s)\int_0^{s}
\beta(\zeta)(\Upsilon(t_n,\zeta)-\Upsilon(t_n,0))  d\zeta ds,\end{equation}
and its boundedness will be shown in Lemma \ref{lamma1} of next section.
Based on the expression   \eqref{UPS}  of $\Upsilon$, it is trivial to get
 $$\Upsilon(t_n,0)=f''(u_n)(v(t_n)^2-(\nabla u(t_n))^2)+\rho  f(u(t_n))-\rho  f'(u(t_n))u(t_n).$$

In a very similar way, one deduces that \begin{equation}\label{  Iv bound2}
\begin{aligned}
  \textmd{  I}^{v}=&\int_0^{h  } \gamma^{\intercal}(-h ) \beta(-s) f\big( \alpha^{\intercal}(s) U(t_n)\big) ds\\
 =&\int_0^{h  } \gamma^{\intercal}(-h ) F(t_n) ds+\int_0^{h  }  \gamma^{\intercal}(-h )
 F'(t_n+\zeta)(h -\zeta) d\zeta\\
 =&\int_0^{h  }  \gamma^{\intercal}(-h ) ds F(t_n) +\int_0^{h  }  (h -s)  \gamma^{\intercal}(2s-h)
 M(0)ds F'(t_n) \\
 &+\int_0^{h  }  (h -s) \gamma^{\intercal}(2s-h)\int_0^{s}
\beta(\zeta)\Upsilon(t_n,\zeta)  d\zeta ds\\
                              =&h ( \cos(h \sqrt{\mathcal{A}}) + \sinc( h \sqrt{\mathcal{A}}) )/2 f\big( u(t_{n})\big)+h ^2/2\sinc(  h \sqrt{\mathcal{A}}) f'\big( u(t_n)\big) v(t_n) \\
 &+h^3\frac{ \sinc( h \sqrt{\mathcal{A}})-\cos(h \sqrt{\mathcal{A}})}{2(h \sqrt{\mathcal{A}})^2} \Upsilon(t_n,0)+R_2(t_n)
\end{aligned}
\end{equation}
with the remainder \begin{equation}\label{R2}R_2(t_n)=\int_0^{h  }  (h -s)  \gamma^{\intercal}(2s-h)\int_0^{s}
\beta(\zeta)(\Upsilon(t_n,\zeta)-\Upsilon(t_n,0))  d\zeta ds.\end{equation}
The estimate of this remainder will be given in  Lemma \ref{lamma2}  of next section.

$\bullet$ \textbf{Approximation to   II$^{u}$ and   II$^{v}$.}

Now we turn to   II$^{u}$ and   II$^{v}$.
For the first one, fourth-order local error  will be derived as follows.
In this paper, it is assumed that  the  nonlinear function $f$ of  \eqref{klein-gordon} satisfies   $\abs{f^{(k)}(w)}\leq C_0$  for $w\in \mathbb{R}$ and $k=1,2,3$. Moreover, we consider   the  regularity condition  $(u(0,x), \partial_tu(0,x))\in H^{1+\max(1,\mu)}(\mathbb{T}^d) \times H^{\max(1,\mu)}(\mathbb{T}^d)$
with $\mu>\frac{d}{2}$.
From
Proposition \ref{lamma00}, it follows that for $d=1$
   \begin{equation}\label{  II bound}
\begin{aligned}
&\norm{\textmd{  II}^{u}}_{H^{1}}\\
\lesssim& \norm{ \int_0^{h  } \frac{ \sin( (h -s)\sqrt{\mathcal{A}})}{\sqrt{\mathcal{A}}} f'\Big(\alpha^{\intercal}(s)U(t_n)\Big) \int_0^{s }  \frac{\sin( ( s-\theta)\sqrt{\mathcal{A}})}{\sqrt{\mathcal{A}}} f\big( u(t_n+\theta)\big)d\theta ds} _{H^{1}}\\
\lesssim &  \int_0^{h  }\abs{h -s}
 \int_0^{  s }\norm{ (  s-\theta)\sinc( (s-\theta)\sqrt{\mathcal{A}}) f'\Big(\alpha^{\intercal}(s)U(t_n)\Big) f\big( u(t_n+\theta)\big)} _{H^{1}}d\theta ds\\
 \lesssim &  h ^4  \max_{s\in[0 ,h ]} \norm{f'\Big(\alpha^{\intercal}(s)U(t_n)\Big)} _{H^{1}} \max_{\zeta\in[0 ,h ]} \norm{f\big( u(t_n+\zeta)\big)} _{H^{1}} \lesssim   h ^4.
\end{aligned}
\end{equation}
 When $d=2$ or $3$, the  Kato–Ponce inequality \eqref{KPI} leads to
$$\norm{\textmd{  II}^{u}}_{H^{1}}
\lesssim  h ^4  \max_{s\in[0 ,h ]} \norm{f'\Big(\alpha^{\intercal}(s)U(t_n)\Big)} _{H^{1}} \max_{\zeta\in[0 ,h ]} \norm{f\big( u(t_n+\zeta)\big)} _{H^{\mu}} \lesssim   h ^4.
$$
However, for  the part $\textmd{  II}^{v}$, we only get the following estimate
   \begin{equation}\label{  IIv app}
\begin{aligned}
 \textmd{  II}^{v}=&\int_0^{h  }\cos( (h -s)\sqrt{\mathcal{A}})f'\Big(\alpha^{\intercal}(s)U(t_n)\Big) \int_0^{s }  \frac{\sin( ( s-\theta)\sqrt{\mathcal{A}})}{\sqrt{\mathcal{A}}} f\big( u(t_n+\theta)\big)d\theta ds\\
=& \int_0^{h  }\cos( (h -s)\sqrt{\mathcal{A}})f'(u(t_n)) \int_0^{s }  \frac{\sin( ( s-\theta)\sqrt{\mathcal{A}})}{\sqrt{\mathcal{A}}}d\theta ds  f\big( u(t_n)\big)+R_3(t_n)\\
=& h^3\frac{ \sinc( h \sqrt{\mathcal{A}})-\cos(h \sqrt{\mathcal{A}})}{2(h \sqrt{\mathcal{A}})^2} f'(u(t_n))  f\big( u(t_n)\big)+R_3(t_n),
\end{aligned}
\end{equation}
where $R_3$ is a remainder defined by
   \begin{equation}\label{R3}
\begin{aligned}
&R_3(t_n)\\
=&\int_0^{h  }\cos( (h -s)\sqrt{\mathcal{A}})f'\Big(\alpha^{\intercal}(s)U(t_n)\Big) \int_0^{s }  \frac{\sin( ( s-\theta)\sqrt{\mathcal{A}})}{\sqrt{\mathcal{A}}} f\big( u(t_n+\theta)\big)d\theta ds\\
&-\int_0^{h  }\cos( (h -s)\sqrt{\mathcal{A}})f'(u(t_n)) \int_0^{s }  \frac{\sin( ( s-\theta)\sqrt{\mathcal{A}})}{\sqrt{\mathcal{A}}}d\theta ds  f\big( u(t_n)\big).
\end{aligned}
\end{equation}
The bound of $R_3(t_n)$  will also be studied in Lemma \ref{lamma2}  of next section.

 $\bullet$ \textbf{Approximation to   III$^{u}$ and   III$^{v}$.}

 Finally, we pay attention to the bounds of   III$^{u}$ and   III$^{v}$ which are respectively presented in \eqref{local errors 1-new} and \eqref{local errorsv 1-new}.
Using Proposition \ref{lamma00} again, the following
estimate holds
   \begin{equation*}
\begin{aligned}
  &\norm{\textmd{  III}^{u}}_{H^{1}}\\\lesssim &\norm{ \int_0^{h  }  \frac{\sin( (h -s)\sqrt{\mathcal{A}})}{\sqrt{\mathcal{A}}} R_{f''}(t_n, s) \Big(\int_0^{s }  \frac{\sin( ( s-\theta)\sqrt{\mathcal{A}})}{\sqrt{\mathcal{A}}} f\big( u(t_n+\theta)\big)d\theta \Big)^2 ds} _{H^{1}}\\
\lesssim &  \int_0^{h  } (h -s)
 \norm{R_{f''}(t_n, s) \Big(\int_0^{s }( s-\theta)  \sinc( ( s-\theta)\sqrt{\mathcal{A}})  f\big( u(t_n+\theta)\big) d\theta \Big)^2} _{H^{1}}ds.
\end{aligned}
\end{equation*}
In view of \eqref{RF2}, we obtain for $d=1$
   \begin{equation}\label{  III bound}
\begin{aligned}
   \norm{\textmd{  III}^{u}}_{H^{1}}    \lesssim&   \int_0^{h} (h -s)
 \norm{R_{f''}(t_n, s)} _{H^{1}} \int_0^{s }( s-\theta)^2    \norm{ f\big( u(t_n+\theta)\big)}^2 _{H^{1}} d\theta  ds\\
 \lesssim  & h ^5    \max_{s\in[0 ,h ]} \norm{R_{f''}(t_n, s)} _{H^{1}}  \max_{\zeta\in[0 ,h ]} \norm{f\big( u(t_n+\zeta)\big)} ^2_{H^{1}}
 \lesssim    h ^5
\end{aligned}
\end{equation}
and  for $d=2$ or $3$
$$   \norm{\textmd{  III}^{u}}_{H^{1}}    \lesssim h ^5    \max_{s\in[0 ,h ]} \norm{R_{f''}(t_n, s)} _{H^{1}}  \max_{\zeta\in[0 ,h ]} \norm{f\big( u(t_n+\zeta)\big)} ^2_{H^{\mu}}
 \lesssim    h ^5.$$
By the same arguments, it is arrived that
   \begin{equation}\label{  IIIv bound}
\begin{aligned}
  &\norm{\textmd{  III}^{v}}_{L^{2}}\lesssim\norm{\textmd{  III}^{v}}_{H^{1}} \\\lesssim &\norm{ \int_0^{h  }   \cos( (h -s)\sqrt{\mathcal{A}}) R_{f''}(t_n, s) \Big(\int_0^{s }  \frac{\sin( ( s-\theta)\sqrt{\mathcal{A}})}{\sqrt{\mathcal{A}}} f\big( u(t_n+\theta)\big)d\theta \Big)^2 ds} _{L^{2}}\\
\lesssim &  \int_0^{h  }
 \norm{R_{f''}(t_n, s) \Big(\int_0^{s }( s-\theta)  \sinc( ( s-\theta)\sqrt{\mathcal{A}})  f\big( u(t_n+\theta)\big) d\theta \Big)^2} _{H^{1}}ds\\
  \lesssim  & h ^4    \max_{s\in[0 ,h ]} \norm{R_{f''}(t_n, s)} _{H^{1}}  \max_{\zeta\in[0 ,h ]} \norm{f\big( u(t_n+\zeta)\big)} ^2_{H^{\max(1,\mu)}}
 \lesssim    h ^4.
\end{aligned}
\end{equation}

As a result, these two parts can be dropped in the numerical scheme without bringing any impact on the accuracy and regularity requirement.

Based on these results, we define the  method for  the Klein--Gordon
equation \eqref{klein-gordon} (Definition \ref{def-se}) by considering \eqref{Duhamel} as well as \eqref{local errors 1-new}-\eqref{local errorsv 1-new} and dropping the remainders $R_1(t_n),R_2(t_n),R_3(t_n)$ and   II$^{u}$,   III$^{u}$,   III$^{v}$ appeared in the above formulation. The construction process of semi-discrete integrator is complete.

\section{Convergence}\label{sec:3}
In this section, we shall derive the convergence  of the proposed  semi-discrete  and fully-discrete integrators.
For each scheme, we will first present the main result and then  prove the error estimates in a  low regularity condition.
\subsection{Convergence of semi-discrete scheme}

\begin{mytheo}  \label{symplectic thm}
Let the  nonlinear function $f$ of the Klein--Gordon equation \eqref{klein-gordon} satisfy   the Lipschitz continuity conditions $\abs{f^{(k)}(w)}\leq C_0$  for $w\in \mathbb{R}$ and $k=1,2,3$.
Under the  regularity condition  \begin{equation}\label{rc exact}(u(0,x), \partial_tu(0,x))\in [H^{1+\max(\mu,1)}(\mathbb{T}^d)\bigcap H^{\max(\mu,1)}_0(\mathbb{T}^d)]\times H^{\max(\mu,1)}(\mathbb{T}^d)\end{equation}
with $\mu>\frac{d}{2}$,  there exist positive constants $C$ and $h_0$ such
that for any $h\in(0,h_0]$
 the numerical result $u_n,v_n$ produced  in Definition \ref{def-se} has the
global error:
 \begin{equation}\label{error bound}
\begin{aligned}
\max_{0\leq n \leq T/h } \norm{u_n-u(t_n)}_{H^{1}} \leq Ch ^3,\ \ \
\max_{0\leq n \leq T/h } \norm{v_n-v(t_n)}_{L^{2}}\leq Ch ^3,
\end{aligned}
\end{equation}
where $C$  depends only on $C_0$ and $T$.
\end{mytheo}

It is noted that for the traditional third-order integrators,  the estimates \eqref{error bound} usually hold under the
the requirement  $$(u(0,x), \partial_tu(0,x))\in  H^{3}(\mathbb{T}^d) \times H^{2}(\mathbb{T}^d).$$
In comparison with this condition, the regularity \eqref{rc exact} is lower.
The proof of Theorem  \ref{symplectic thm} is given in the rest part of this section.  We begin with  the bounds on the remainders $R_1(t_n)$, $R_2(t_n)$ and   $R_3(t_n)$ which are presented in  \eqref{R1}, \eqref{R2} and \eqref{R3}, respectively. 
\begin{mylemma}\label{lamma1}
Under the  conditions of Theorem   \ref{symplectic thm},
the remainder $R_1(t_n)$ given in \eqref{R1}  is bounded by
   \begin{equation}\label{R1 bound result}
\begin{aligned}\norm{R_1(t_n)}_{H^{1}}\lesssim h^4.
\end{aligned}
\end{equation}
\end{mylemma}
\begin{proof}
By using the expression \eqref{R1},  $R_1(t_n)$ can be estimated as
   \begin{equation}\label{R1RR}
\begin{aligned}
&\norm{R_1(t_n)}_{H^{1}} \\ =  &  \norm{\int_0^{h  } \int_0^{s} (h -s)  \frac{ \sin( (h-2s+\zeta)\sqrt{\mathcal{A}}) }{\sqrt{\mathcal{A}}}\big( \Upsilon(t_n,\zeta)-\Upsilon(t_n,0)\big)d\zeta ds}_{H^{1}} \\
\lesssim&   \int_0^{h  } (h -s) \int_0^{s} \norm{    \sin( (h-2s+\zeta)\sqrt{\mathcal{A}})   \frac{\Upsilon(t_n,\zeta)-\Upsilon(t_n,0) }{\sqrt{\mathcal{A}}} }_{H^{1}}d\zeta ds \\
\lesssim&   \int_0^{h  } (h -s) \int_0^{s}  \norm{ \frac{\Upsilon(t_n,\zeta)-\Upsilon(t_n,0) }{\sqrt{\mathcal{A}}} }_{H^{1}}d\zeta ds\\
\lesssim&  h ^3  \max_{\zeta \in[0,h ]} \norm{\frac{ \Upsilon(t_n,\zeta)-\Upsilon(t_n,0)}{\sqrt{\mathcal{A}}}}_{H^{1}}\\
\lesssim &  h ^3  \max_{\zeta \in[0,h ]} \norm{\Upsilon(t_n,\zeta)-\Upsilon(t_n,0)}_{L^{2}}.
\end{aligned}
\end{equation}
Then  we are devoted to the bound of $\Upsilon(t_n,\zeta)-\Upsilon(t_n,0)$. According to the result \eqref{UPS} of  $\Upsilon$, we decompose  $\Upsilon(t_n,\zeta)-\Upsilon(t_n,0)$ into four parts
   \begin{equation}\label{mUpsilon}
\begin{aligned}
&\Upsilon(t_n,\zeta)-\Upsilon(t_n,0)\\
=&\underbrace{f''\big( \alpha^{\intercal}(\zeta) U(t_n)\big)\big(\gamma^{\intercal}(-\zeta) U(t_n) \big)^2-
f''\big( \alpha^{\intercal}(0) U(t_n)\big)\big(\gamma^{\intercal}(0) U(t_n) \big)^2}_{=:\Delta_1 f}\\&
+ \underbrace{f''\big( \alpha^{\intercal}(0) U(t_n)\big)\big(\nabla \alpha^{\intercal}(0) U(t_n) \big)^2-  f''\big( \alpha^{\intercal}(\zeta) U(t_n)\big)\big(\nabla \alpha^{\intercal}(\zeta) U(t_n) \big)^2}_{=:\Delta_2 f}\\
                                     &+\underbrace{\rho  f\big( \alpha^{\intercal}(\zeta) U(t_n)\big)-\rho  f\big( \alpha^{\intercal}(0) U(t_n)\big)}_{=:\Delta_3 f}\\
                                     &+\underbrace{\rho  f'\big( \alpha^{\intercal}(0) U(t_n)\big)\alpha^{\intercal}(0) U(t_n)-\rho  f'\big( \alpha^{\intercal}(\zeta) U(t_n)\big)\alpha^{\intercal}(\zeta) U(t_n)}_{=:\Delta_4 f}.
                                     \end{aligned}
\end{equation}
In what follows, we  will deduce the results for these four terms one by one.

 $\bullet$ \textbf{Bound on $\norm{\Delta_1 f}_{L^{2}}$.}

We first recall the Kato–Ponce inequalities given in Lemma \ref{lamma0} and as a result
   \begin{equation*}
\begin{aligned}
\norm{\Delta_1 f}_{L^{2}}
\lesssim&\norm{f''\big( \alpha^{\intercal}(\zeta) U(t_n)\big)-f''\big( \alpha^{\intercal}(0) U(t_n)\big)}_{H^{\mu}}\norm{\big(\gamma^{\intercal}(-\zeta) U(t_n) \big)^2}_{L^{2}}\\
&+\norm{f''\big( \alpha^{\intercal}(0) U(t_n)\big)}_{H^{\mu}}\norm{\big(\gamma^{\intercal}(-\zeta) U(t_n) \big)^2-\big(\gamma^{\intercal}(0) U(t_n) \big)^2}_{L^{2}}\\
\lesssim&\norm{f''\big( \alpha^{\intercal}(\zeta) U(t_n)\big)-f''\big( \alpha^{\intercal}(0) U(t_n)\big)}_{H^{\mu}}\norm{\big(\gamma^{\intercal}(-\zeta) U(t_n) \big)^2}_{L^{2}}\\
&+\norm{f''\big( \alpha^{\intercal}(0) U(t_n)\big)}_{H^{\mu}}\norm{ \gamma^{\intercal}(-\zeta) U(t_n) +\gamma^{\intercal}(0) U(t_n) }_{H^{\mu}}\\&\ \ \ \norm{ \gamma^{\intercal}(-\zeta) U(t_n)  - \gamma^{\intercal}(0) U(t_n)  }_{L^{2}}.\end{aligned}
\end{equation*}
Based on this result, it is needed to estimate
   \begin{equation*}
\begin{aligned}
&\norm{f''\big( \alpha^{\intercal}(\zeta) U(t_n)\big)-f''\big( \alpha^{\intercal}(0) U(t_n)\big)}_{H^{\mu}}\\
\lesssim&\norm{  \alpha^{\intercal}(\zeta) U(t_n)- \alpha^{\intercal}(0) U(t_n)}_{H^{\mu}}\\
 \lesssim&\norm{  \cos(\zeta\sqrt{\mathcal{A}}) u(t_n)-u(t_n)+\zeta \sinc( \zeta\sqrt{\mathcal{A}})  v(t_n)}_{H^{\mu}}
 \\
 \lesssim&\norm{  -2\sin(\zeta\sqrt{\mathcal{A}}/2) \sinc(\zeta\sqrt{\mathcal{A}}/2) \zeta\sqrt{\mathcal{A}} u(t_n)}_{H^{\mu}}+\zeta\norm{ v(t_n)}_{H^{\mu}} \\
 \lesssim& \zeta\big(\norm{  u(t_n)}_{H^{1+\mu}}+\norm{ v(t_n)}_{H^{\mu}}\big)\end{aligned}
\end{equation*}
by using Proposition \ref{lamma00}.  Then the Sobolev embedding theorem shows that
   \begin{equation*}
\begin{aligned}
&\norm{ (\gamma^{\intercal}(-\zeta) U(t_n))^2}_{L^{2}}\\
 \lesssim&  \norm{ \gamma^{\intercal}(-\zeta) U(t_n)}^2_{L^{4}}
  \lesssim  \norm{ \gamma^{\intercal}(-\zeta) U(t_n)}^2_{W^{1,p}}  \quad (W^{1,p}\hookrightarrow L^{4})\\
   \lesssim& \norm{ \gamma^{\intercal}(-\zeta) U(t_n)}^2_{W^{\frac{d}{4},2}}\quad   (W^{\frac{d}{4},2}\hookrightarrow W^{1,p})\\
 \lesssim &\norm{ \gamma^{\intercal}(-\zeta) U(t_n)}^2_{H^{\frac{d}{4}}}
 \lesssim  \norm{-\sqrt{\mathcal{A}} \sin(\zeta\sqrt{\mathcal{A}}) u(t_n)+\cos(\zeta\sqrt{\mathcal{A}}) v(t_n)}^2_{H^{\frac{d}{4}}}\\
  \lesssim&  \norm{  u(t_n)}^2_{H^{1+\frac{d}{4}}}+\norm{ v(t_n)}^2_{H^{\frac{d}{4}}}+2\norm{  u(t_n)}_{H^{1+\frac{d}{4}}}\norm{ v(t_n)}_{H^{\frac{d}{4}}}.
\end{aligned}
\end{equation*}
Meanwhile, it follows from \eqref{notations w} that
   \begin{equation*}
\begin{aligned}
 &\norm{ \gamma^{\intercal}(-\zeta) U(t_n) +\gamma^{\intercal}(0) U(t_n) }_{H^{\mu}} \\\lesssim& \norm{-\sqrt{\mathcal{A}} \sin(\zeta\sqrt{\mathcal{A}}) u(t_n)+\cos(\zeta\sqrt{\mathcal{A}}) v(t_n)+v(t_n)} _{H^{\mu}}\\
 \lesssim& \norm{ u(t_n)} _{H^{1+\mu}}+\norm{ v(t_n)} _{H^{\mu}}.
\end{aligned}
\end{equation*}
Analogously, we obtain
   \begin{equation*}
\begin{aligned}
 &\norm{ \gamma^{\intercal}(-\zeta) U(t_n) -\gamma^{\intercal}(0) U(t_n) }_{L^{2}}\\ \lesssim& \norm{-\sqrt{\mathcal{A}} \sin(\zeta\sqrt{\mathcal{A}}) u(t_n)+\cos(\zeta\sqrt{\mathcal{A}}) v(t_n)-v(t_n)} _{L^{2}}\\
\lesssim& \norm{\zeta\sqrt{\mathcal{A}} \sqrt{\mathcal{A}} \sinc(\zeta\sqrt{\mathcal{A}}) u(t_n)} _{L^{2}}+\norm{ -2\sin(\zeta\sqrt{\mathcal{A}}/2) \sinc(\zeta\sqrt{\mathcal{A}}/2) \zeta\sqrt{\mathcal{A}}  v(t_n)} _{L^{2}}\\
  \lesssim& \zeta \norm{ u(t_n)} _{H^{2}}+ \zeta \norm{   v(t_n)} _{H^{1}}.
\end{aligned}
\end{equation*}
Under the regularity condition
$$(u(0,x), \partial_tu(0,x))\in  H^{2}(\mathbb{T}^d) \times H^{1}(\mathbb{T}^d),$$
one gets
$$\norm{ \gamma^{\intercal}(-\zeta) U(t_n) -\gamma^{\intercal}(0) U(t_n) }_{L^{2}}   \lesssim  \zeta.$$

 Overall, collecting all the estimates together yields
   \begin{equation*}
\begin{aligned}
\norm{\Delta_1 f}_{L^{2}}
\lesssim& \zeta\big(\norm{  u(t_n)}_{H^{1+\mu}}+\norm{ v(t_n)}_{H^{\mu}}\big) \big(\norm{  u(t_n)}^2_{H^{1+\frac{d}{4}}}+\norm{ v(t_n)}^2_{H^{\frac{d}{4}}}\\&+2\norm{  u(t_n)}_{H^{1+\frac{d}{4}}}\norm{ v(t_n)}_{H^{\frac{d}{4}}} + \norm{ u(t_n)} _{H^{2}}+  \norm{   v(t_n)} _{H^{1}}\big)
\lesssim\zeta.\end{aligned}
\end{equation*}

 $\bullet$ \textbf{Bound on $\norm{\Delta_2 f}_{L^{2}}$.}

By the similar argument as $\Delta_1 f$,  we decompose $\Delta_2 f$ into four terms:
   \begin{equation*}
\begin{aligned}
&\norm{\Delta_2 f}_{L^{2}}\\
=&\norm{f''\big( \alpha^{\intercal}(\zeta) U(t_n)\big)\big(\nabla \alpha^{\intercal}(\zeta) U(t_n) \big)^2-
f''\big( \alpha^{\intercal}(0) U(t_n)\big)\big(\nabla \alpha^{\intercal}(0) U(t_n) \big)^2}_{L^{2}}\\
\lesssim&\norm{f''\big( \alpha^{\intercal}(\zeta) U(t_n)\big)-f''\big( \alpha^{\intercal}(0) U(t_n)\big)}_{H^{\mu}}\norm{\big(\nabla \alpha^{\intercal}(\zeta) U(t_n) \big)^2}_{L^{2}}\\
&+\norm{f''\big( \alpha^{\intercal}(0) U(t_n)\big)}_{H^{\mu}}\norm{\big(\nabla \alpha^{\intercal}(\zeta) U(t_n) \big)^2-\big(\nabla \alpha^{\intercal}(0) U(t_n) \big)^2}_{L^{2}}.\end{aligned}
\end{equation*}
It is noted that this formula contains the same expression as   $\Delta_1 f$ and thus we only need to estimate the different two terms: $\norm{\big(\nabla \alpha^{\intercal}(\zeta) U(t_n) \big)^2}_{L^{2}}$ and $\norm{\big(\nabla \alpha^{\intercal}(\zeta) U(t_n) \big)^2-\big(\nabla \alpha^{\intercal}(0) U(t_n) \big)^2}_{L^{2}}$. For the first one, it follows from  the Sobolev embedding theorem
that
    \begin{equation*}
\begin{aligned}
&\norm{ (\nabla \alpha^{\intercal}(\zeta) U(t_n))^2}_{L^{2}}\\
 \lesssim&  \norm{\nabla \alpha^{\intercal}(\zeta) U(t_n)}^2_{L^{4}}
  \lesssim  \norm{ \nabla \alpha^{\intercal}(\zeta) U(t_n)}^2_{W^{1,p}}  \quad (W^{1,p}\hookrightarrow L^{4})\\
   \lesssim& \norm{ \nabla \alpha^{\intercal}(\zeta) U(t_n)}^2_{W^{\frac{d}{4},2}}\quad   (W^{\frac{d}{4},2}\hookrightarrow W^{1,p})\\
 \lesssim &\norm{ \nabla \alpha^{\intercal}(\zeta) U(t_n)}^2_{H^{\frac{d}{4}}}
 \lesssim  \norm{\cos(\zeta\sqrt{\mathcal{A}}) u(t_n)+  \sin( \zeta\sqrt{\mathcal{A}})\frac{v(t_n)}{\sqrt{\mathcal{A}}}}^2_{H^{1+\frac{d}{4}}}\\
  \lesssim&  \norm{  u(t_n)}^2_{H^{1+\frac{d}{4}}}+\norm{ v(t_n)}^2_{H^{\frac{d}{4}}}+2\norm{  u(t_n)}_{H^{1+\frac{d}{4}}}\norm{ v(t_n)}_{H^{\frac{d}{4}}}.
\end{aligned}
\end{equation*}
 For the second one, we use the Kato–Ponce inequality to get
     \begin{equation*}
\begin{aligned}
&\norm{\big(\nabla \alpha^{\intercal}(\zeta) U(t_n) \big)^2-\big(\nabla \alpha^{\intercal}(0) U(t_n) \big)^2}_{L^{2}}  \\
\lesssim& \norm{\nabla \alpha^{\intercal}(\zeta) U(t_n) +\nabla \alpha^{\intercal}(0) U(t_n) }_{H^{\mu}}\norm{\nabla \alpha^{\intercal}(\zeta) U(t_n) -\nabla \alpha^{\intercal}(0) U(t_n) }_{L^{2}}\\
\lesssim &\norm{\cos(\zeta\sqrt{\mathcal{A}}) u(t_n)+  \sin( \zeta\sqrt{\mathcal{A}})\frac{v(t_n)}{\sqrt{\mathcal{A}}}+u(t_n)} _{H^{1+\mu}}\\& \norm{(\cos(\zeta\sqrt{\mathcal{A}})-1) u(t_n)+\zeta  \sinc( \zeta\sqrt{\mathcal{A}}) v(t_n) } _{H^{1}}\\
\lesssim &\big(\norm{ u(t_n)} _{H^{1+\mu}}+\norm{ v(t_n)} _{H^{\mu}}\big)\\& \big(\norm{ -2\sin(\zeta\sqrt{\mathcal{A}}/2) \sinc(\zeta\sqrt{\mathcal{A}}/2) \zeta\sqrt{\mathcal{A}}  u(t_n)} _{H^{1}}+\norm{\zeta\sinc(\zeta\sqrt{\mathcal{A}}) v(t_n)} _{H^{1}}\big)\\
\lesssim & \zeta \big(\norm{ u(t_n)} _{H^{1+\mu}}+\norm{ v(t_n)} _{H^{\mu}}\big) \big(\norm{ u(t_n)} _{H^{2}}+ \zeta \norm{   v(t_n)} _{H^{1}}\big).
\end{aligned}
\end{equation*}
 To summarize, we have
$$
\norm{\Delta_2 f}_{L^{2}} \lesssim \zeta.$$

  $\bullet$ \textbf{Bound on $\norm{\Delta_3 f}_{L^{2}}$.}

The following  estimate can be proved in the same way as above
     \begin{equation*}
\begin{aligned}
 &\norm{\Delta_3 f}_{L^{2}}\\
 =& \norm{ \rho  f\big( \alpha^{\intercal}(\zeta) U(t_n)\big)-\rho  f\big( \alpha^{\intercal}(0) U(t_n)\big)}_{L^{2}}\\
   \lesssim& \norm{  f'\Big( \big(\alpha^{\intercal}(\zeta)  -\varsigma(\alpha^{\intercal}(0) -\alpha^{\intercal}(\zeta) )\big)U(t_n)\Big)
    (\alpha^{\intercal}(0) -\alpha^{\intercal}(\zeta) )U(t_n)  }_{L^{2}}\\
       \lesssim& \norm{  f'\Big( \big(\alpha^{\intercal}(\zeta)  -\varsigma(\alpha^{\intercal}(0) -\alpha^{\intercal}(\zeta) )\big)U(t_n)\Big)}_{H^{\mu}} \norm{
    (\alpha^{\intercal}(0) -\alpha^{\intercal}(\zeta) )U(t_n)}_{L^{2}}\\
 \lesssim &\norm{  \alpha^{\intercal}(\zeta) U(t_n)- \alpha^{\intercal}(0) U(t_n)}_{L^{2}}
  \lesssim \zeta\big(\norm{  u(t_n)}_{H^{1}}+\norm{ v(t_n)}_{L^{2}}\big),
  \end{aligned}
\end{equation*}
where $\varsigma\in[0,1]$.

   $\bullet$ \textbf{Bound on $\norm{\Delta_4 f}_{L^{2}}$.}

   For the last part, we analogously obtain
        \begin{equation*}
\begin{aligned}
 \norm{\Delta_4 f}_{L^{2}}=& \norm{ \rho  f'\big( \alpha^{\intercal}(0) U(t_n)\big)\alpha^{\intercal}(0) U(t_n)-\rho  f'\big( \alpha^{\intercal}(\zeta) U(t_n)\big)\alpha^{\intercal}(\zeta) U(t_n)}_{L^{2}}\\
\lesssim&\norm{f'\big( \alpha^{\intercal}(\zeta) U(t_n)\big)-f'\big( \alpha^{\intercal}(0) U(t_n)\big)}_{H^{\mu}}\norm{ \alpha^{\intercal}(\zeta) U(t_n)}_{L^{2}}\\
&+\norm{f'\big( \alpha^{\intercal}(0) U(t_n)\big)}_{H^{\mu}}\norm{\alpha^{\intercal}(\zeta) U(t_n) - \alpha^{\intercal}(0) U(t_n)}_{L^{2}}\\
\lesssim&\zeta\big(\norm{  u(t_n)}_{H^{1+\mu}}+\norm{ v(t_n)}_{H^{\mu}}\big)\big(\norm{  u(t_n)}_{H^{1}}+\zeta\norm{ v(t_n)}_{L^{2}}\big).
  \end{aligned}
\end{equation*}
 Consequently, based on the above results, we arrive at
 \begin{equation}\label{zeta- bound result}\max_{\zeta \in[0,h ]} \norm{\Upsilon(t_n,\zeta)-\Upsilon(t_n,0)}_{L^{2}}\lesssim\max_{\zeta \in[0,h ]}   \zeta \lesssim h.\end{equation}
 Combining this with \eqref{R1RR} immediately gives the statement \eqref{R1 bound result}  and  the proof of this lemma is complete.

\end{proof}

\begin{mylemma}\label{lamma2}
Suppose that the  conditions of Theorem   \ref{symplectic thm} hold. For
the remainders $R_2(t_n)$ and $R_3(t_n)$   respectively defined in \eqref{R2} and \eqref{R3}, they are bounded by
   \begin{equation}\label{R2 bound result}
\begin{aligned}\norm{R_2(t_n)}_{L^{2}}\lesssim h^4,\ \ \ \norm{R_3(t_n)}_{L^{2}}\lesssim h^4.
\end{aligned}
\end{equation}
\end{mylemma}
\begin{proof}
Keeping \eqref{zeta- bound result} in mind, the following  estimate can be proved in the same way  as Lemma \ref{lamma1}
\begin{equation*}
\begin{aligned}
&\norm{R_2(t_n)}_{L^{2}}\\= &\norm{\int_0^{h  } \int_0^{s} (h -s)  \cos( (h-2s+\zeta)\sqrt{\mathcal{A}})  \big( \Upsilon(t_n,\zeta)-\Upsilon(t_n,0)\big)d\zeta ds}_{L^{2}} \\
\lesssim& \int_0^{h  } \int_0^{s} \abs{h -s} \norm{ \big( \Upsilon(t_n,\zeta)-\Upsilon(t_n,0)\big)}_{L^{2}}d\zeta ds \\
\lesssim&  h ^3  \max_{\zeta \in[0,h ]} \norm{ \Upsilon(t_n,\zeta)-\Upsilon(t_n,0) }_{L^{2}}\lesssim h^4.
\end{aligned}
\end{equation*}
Then we  represent the scheme of $R_3(t_n)$ as
   \begin{equation*}
\begin{aligned}
&R_3(t_n)\\
=&\int_0^{h  }\cos( (h -s)\sqrt{\mathcal{A}})f'\big(\alpha^{\intercal}(s)U(t_n)\big) \int_0^{s }  \frac{\sin( ( s-\theta)\sqrt{\mathcal{A}})}{\sqrt{\mathcal{A}}} f\big( u(t_n+\theta)\big)d\theta ds\\
&-\int_0^{h  }\cos( (h -s)\sqrt{\mathcal{A}})f'(u(t_n)) \int_0^{s }  \frac{\sin( ( s-\theta)\sqrt{\mathcal{A}})}{\sqrt{\mathcal{A}}}d\theta ds  f\big( u(t_n)\big)
\end{aligned}
\end{equation*}
and split it into
   \begin{equation}\label{R3n}
\begin{aligned}
R_3(t_n)
=&\int_0^{h  }\cos( (h -s)\sqrt{\mathcal{A}})\big(f'\big(\alpha^{\intercal}(s)U(t_n)\big)-f'(u(t_n))\big)\\& \int_0^{s }  \frac{\sin( ( s-\theta)\sqrt{\mathcal{A}})}{\sqrt{\mathcal{A}}} f\big( u(t_n+\theta)\big)d\theta ds+\int_0^{h  }\cos( (h -s)\sqrt{\mathcal{A}})f'(u(t_n)) \\
&\int_0^{s }  \frac{\sin( ( s-\theta)\sqrt{\mathcal{A}})}{\sqrt{\mathcal{A}}} \big(f\big( u(t_n+\theta)\big)-f\big( u(t_n)\big)\big)d\theta ds.
\end{aligned}
\end{equation}
Some derivations lead to
    \begin{equation*}
\begin{aligned}
&\norm{R_3(t_n)}_{L^{2}}\\
\lesssim &  \int_0^{h  }\int_0^{ s } \norm{ f'\big(\alpha^{\intercal}(s)U(t_n)\big)-f'(u(t_n))} _{H^{\mu}}  \abs{s-\theta}\norm{ f\big( u(t_n+\theta)\big)} _{L^{2}}d\theta ds \\
&+  \int_0^{h  }\int_0^{ s } \norm{ f'(u(t_n))} _{H^{\mu}}  \abs{s-\theta}\norm{ f\big( u(t_n+\theta)\big)-f\big( u(t_n)\big)} _{L^{2}}d\theta ds \\
\lesssim &  h ^3  \max_{s\in[0,h]} \norm{ f'\big(\alpha^{\intercal}(s)U(t_n)\big)-f'(u(t_n))} _{H^{\mu}}\\
&+h ^3  \max_{\theta\in[0,h]} \norm{f\big( u(t_n+\theta)\big)-f\big( u(t_n)\big)} _{L^{2}}\\
\lesssim &  h ^3  \max_{s\in[0,h]} \norm{ f''\big(\varsigma_1(\alpha^{\intercal}(s)U(t_n)-u(t_n))+u(t_n) \big)\big(\alpha^{\intercal}(s)U(t_n)-u(t_n)\big)} _{H^{\mu}}\\
&+h ^3  \max_{\theta\in[0,h]} \norm{ f'\big(\varsigma_2( u(t_n+\theta)-u(t_n))+u(t_n) \big)\big( u(t_n+\theta)-u(t_n)\big)}_{L^{2}}\\
\lesssim &  h ^3  \max_{s\in[0,h]} \norm{ f''\big(\varsigma_1(\alpha^{\intercal}(s)U(t_n)-u(t_n))+u(t_n) \big)} _{H^{\mu}}\norm{\alpha^{\intercal}(s)U(t_n)-u(t_n)} _{H^{\mu}}\\
&+h ^3  \max_{\theta\in[0,h]} \norm{ f'\big(\varsigma_2( u(t_n+\theta)-u(t_n))+u(t_n) \big)} _{H^{\mu}}\norm{ u(t_n+\theta)-u(t_n)}_{L^{2}}\\
\lesssim &  h ^3  \max_{s\in[0,h]} \norm{  \alpha^{\intercal}(s)U(t_n) -u(t_n)} _{H^{\mu}}+h ^3  \max_{\theta\in[0,h]} \norm{ u(t_n+\theta) -  u(t_n)} _{L^{2}}\\
 \lesssim&  h ^3  \max_{s\in[0,h]} s\big(\norm{  u(t_n)}_{H^{1+\mu}}+\norm{ v(t_n)}_{H^{\mu}}\big) + h ^3  \max_{\theta\in[0,h]} \theta  \norm{  v(t_n)}_{L^{2}}
  \lesssim  h ^4,
\end{aligned}
\end{equation*}
with $\varsigma_1, \varsigma_2\in[0,1]$.
The proof is complete.
\end{proof}

Sofar we have derived the bounds for the remainders which are dropped in the numerical scheme.
Based on them, we can present the error analysis (the proof of   Theorem \ref{symplectic thm}) in what follows.

\textbf{Proof of   Theorem \ref{symplectic thm}.}\begin{proof}
To estimate the error of the scheme \eqref{method}
$$e_n^{u}:=u(t_{n})-u_n,\quad e_n^{v}:=v(t_{n})-v_n,\quad    0\leq n\leq  T /h ,$$
we shall first consider the local truncation errors which are defined by inserting
the solution of  \eqref{klein-gordon} into \eqref{method}:
 \begin{equation}\label{local errors}
\begin{aligned}
\zeta_n^{u}:=&u(t_{n+1})-\cos(h \sqrt{\mathcal{A}}) u(t_n)-h  \sinc( h \sqrt{\mathcal{A}})  v(t_n)-h^2\Phi_1(h \sqrt{\mathcal{A}}) f( u(t_n))\\&-h^3  \Psi_1(h \sqrt{\mathcal{A}})f'( u(t_n))v(t_n) -h^4  \Psi_2(h \sqrt{\mathcal{A}})  \Big(f''(u(t_n))v(t_n)^2\\&-  f''(u(t_n))(\nabla u(t_n))^2+\rho  f(u(t_n))-\rho  f'(u(t_n))u(t_n)\Big),\\
\zeta_n^{v}:=&v(t_{n+1})+h \mathcal{A} \sinc( h \sqrt{\mathcal{A}}) u(t_n)-\cos(h \sqrt{\mathcal{A}}) v(t_n)-
h  \Phi_2(h \sqrt{\mathcal{A}})  f( u(t_n)) \\&-h^2\Phi_1(h \sqrt{\mathcal{A}})f'( u(t_n))v(t_n) -h^3  \Psi_1(h \sqrt{\mathcal{A}})  \Big(f''(u(t_n))(v(t_n)^2\\&-  (\nabla u(t_n))^2)+\rho  f(u(t_n))- f'(u(t_n))(\rho u(t_n)- f(u(t_n)))\Big),
\end{aligned}
\end{equation}
where $n=0,1,\ldots, T/h-1$.  For the coefficient functions appeared above, it follows from  \eqref{method-coefficient} that the  functions $\Phi_1(m),\Phi_2(m),m\Psi_1(m),m^2\Psi_2(m)$ are uniformly bounded for any
$m\in \mathbb{R}$ which immediately leads to
\begin{equation}\label{method-coefficient-bound}
\begin{aligned}
 &\norm{\Phi_1(h \sqrt{\mathcal{A}}) y}_{H^{\nu}}\lesssim \norm{y}_{H^{\nu}},\qquad \ \ \  \norm{\Phi_2(h \sqrt{\mathcal{A}}) y}_{H^{\nu}}\lesssim \norm{y}_{H^{\nu}},\\
   &\norm{h \sqrt{\mathcal{A}} \Psi_1(h \sqrt{\mathcal{A}}) y}_{H^{\nu}}\lesssim \norm{y}_{H^{\nu}},\ \ \norm{(h \sqrt{\mathcal{A}})^2\Psi_2(h \sqrt{\mathcal{A}}) y}_{H^{\nu}}\lesssim \norm{y}_{H^{\nu}},
\end{aligned}
\end{equation}
where we assume that $y\in H^{\nu}$ with any $\nu\geq 0$.

Subtracting the corresponding local error terms \eqref{local errors} from the scheme  \eqref{method}, we get the recurrence relation
for the errors
 \begin{equation}\label{equ errors}
\begin{aligned}
&e_{n+1}^{u}- \cos(h \sqrt{\mathcal{A}})e_n^{u}-h  \sinc( h \sqrt{\mathcal{A}})e_{n}^{v}=\zeta_n^{u}+h^2\eta_n^{u},\\&
e_{n+1}^{v}+ h \mathcal{A} \sinc( h \sqrt{\mathcal{A}})e_n^{u}-\cos(h \sqrt{\mathcal{A}})e_n^{v}=\zeta_n^{v}+h\eta_n^{v},\ n=0,1,\ldots, T/h-1,
\end{aligned}
\end{equation}
where   we denote
 \begin{equation}\label{sta errors}
\begin{aligned}
\eta_n^{u}:=&\Phi_1(h \sqrt{\mathcal{A}}) \big(f( u_n)-f( u(t_n))\big) +h   \Psi_1(h \sqrt{\mathcal{A}})\big(f'( u_n)v_n-f'( u(t_n))v(t_n)\big)\\& +h^2  \Psi_2(h \sqrt{\mathcal{A}})  \Big(\big(f''(u_n)v_n^2-f''(u(t_n))v(t_n)^2\big) \\&-  \big(f''(u_n)(\sqrt{\mathcal{A}}u_n)^2-f''(u(t_n))(\nabla u(t_n))^2\big)\\
&+\rho  \big(f(u_n)-f(u(t_n))\big)-\rho  \big(f'(u_n)u_n-f'(u(t_n))u(t_n)\big)\Big),\\
 \eta_n^{v}:=& \Phi_2(h \sqrt{\mathcal{A}})   \big(f( u_n)-f( u(t_n))\big) +h\Phi_1(h \sqrt{\mathcal{A}})\big(f'( u_n)v_n-f'( u(t_n))v(t_n)\big)\\&  +h^2  \Psi_1(h \sqrt{\mathcal{A}})  \Big(\big(f''(u_n)v_n^2-f''(u(t_n))v(t_n)^2\big) \\&-  \big(f''(u_n)(\sqrt{\mathcal{A}}u_n)^2-f''(u(t_n))(\nabla u(t_n))^2\big)\\
&+\rho  \big(f(u_n)-f(u(t_n))\big)-\rho  \big(f'(u_n)u_n-f'(u(t_n))u(t_n)\big) \\&-  \big( f'(u_n) f(u_n)- f'(u(t_n)) f(u(t_n))\big)\Big).
\end{aligned}
\end{equation}
The starting value of \eqref{equ errors} is  $e_0^{u}=e_0^{v}=0.$

In what follows, the proof is  divided into  three parts. The first one is about   the boundedness of numerical solution,
the second is devoted to local errors $\zeta_n^{u}, \zeta_n^{v}$ and stability $\eta_n^{u}, \eta_n^{v}$, and the last one concerns global errors $e_n^{u}, e_n^{v}$.

$\bullet$ In this part, we prove the boundedness of numerical solution:
\begin{equation}\label{bound num}
 \norm{u_{n} } _{H^{1+\max(\mu,1)}}+\norm{v_{n}}_{H^{\max(\mu,1)}}\lesssim 1,\ \ \ n=0,1,\ldots, T/h.
\end{equation}
For $n=0$,   (\ref{bound num})  is obviously true. Then we assume  that it holds  up to some $n=1,2,\ldots,m$, and we shall show that (\ref{bound num}) is true for $m+1$.
To this end, we reformulate
the scheme   \eqref{method}
 as
\begin{equation*}
\begin{aligned}
&\left(
  \begin{array}{c}
u_{m+1}\\
v_{m+1} \\
  \end{array}
\right)=\left(
  \begin{array}{cc}
         \cos(h \sqrt{\mathcal{A}})       & h  \sinc( h \sqrt{\mathcal{A}}) \\
    - h \mathcal{A} \sinc( h \sqrt{\mathcal{A}})  &  \cos(h \sqrt{\mathcal{A}}) \\
  \end{array}
\right) \left(
  \begin{array}{c}
   u_{m} \\
     v_{m} \\
  \end{array}
\right)
\\&+h \left(
          \begin{array}{c}
            h \Phi_1  f( u_m) +h^2 \Psi_1 f'( u_m)v_m +h^3  \Psi_2  F_1(u_m,v_m)\\
            \Phi_2   f( u_m)  +h\Phi_1 f'( u_m)v_m +h^2  \Psi_1  \big(F_1(u_m,v_m)+f'(u_m) f(u_m) \big)\\
          \end{array}
        \right).
\end{aligned}
\end{equation*}
Using the notation $\norm{\left(
  \begin{array}{c}
  w_1\\
   w_2 \\
  \end{array}
\right)}_{1+\max(\mu,1)}:=\sqrt{\norm{ w_1}_{H^{1+\max(\mu,1)}}^2+\norm{w_2}_{H^{\mu}}^2}$, it is derived that
  $$\norm{\left(
  \begin{array}{cc}
         \cos(h \sqrt{\mathcal{A}})       & h  \sinc( h \sqrt{\mathcal{A}}) \\
    - h \mathcal{A} \sinc( h \sqrt{\mathcal{A}})  &  \cos(h \sqrt{\mathcal{A}}) \\
  \end{array}
\right) \left(
  \begin{array}{c}
    u_m \\
  v_m \\
  \end{array}
\right) }_{1+\max(\mu,1)}\leq\norm{ \left(
  \begin{array}{c}
    u_m \\
  v_m \\
  \end{array}
\right) }_{1+\max(\mu,1)}.$$
According to the boundedness   (\ref{bound num}) up to  $m$, the expression \eqref{method-coefficient} of the coefficient functions $ \Psi_1, \Psi_2, \Phi_1, \Phi_2$  and the assumption on $f$ given in Theorem  \ref{symplectic thm}, it is easy to see that
\begin{equation*}
\begin{aligned}
&\norm{\left(
  \begin{array}{c}
u_{m+1}\\
v_{m+1} \\
  \end{array}
\right)}_{1+\max(\mu,1)}\leq \norm{\left(
  \begin{array}{c}
u_{m}\\
v_{m} \\
  \end{array}
\right)}_{1+\max(\mu,1)}
+   Ch .
\end{aligned}
\end{equation*}
Applying the Gronwall inequality  immediately yields  (\ref{bound num}) for $n=m+1$.

$\bullet$ By the bounds \eqref{ II bound}, \eqref{ III bound}-\eqref{ IIIv bound} and Lemmas \ref{lamma1}--\ref{lamma2},
the local errors $\zeta_n^{u}, \zeta_n^{v}$ are estimated as
\begin{equation*}
\begin{aligned}&\norm{\zeta_n^{u}}_{H^1}\lesssim  \norm{R_1(t_n)}_{H^1} +\norm{\textmd{Part II}^{u}}_{H^{1}}+\norm{\textmd{Part III}^{u}}_{H^{1}} \lesssim  h ^4,\\ &\norm{\zeta_n^{v}}_{L^2}\lesssim  \norm{R_2(t_n)}_{L^2} +\norm{R_3(t_n)}_{L^2}+\norm{\textmd{Part III}^{v}}_{L^2} \lesssim  h ^4.\end{aligned}
\end{equation*}
 To establish the estimate for  the stability  terms $\eta_n^{u}$ and $\eta_n^{v}$ defined in \eqref{sta errors},
we  consider the   terms of \eqref{sta errors} separately. For the first and last terms, it is easy to see that
 \begin{equation*}
\begin{aligned}
&h^2  \norm{\Psi_2(h \sqrt{\mathcal{A}}) \Big(\rho  \big(f(u_n)-f(u(t_n))\big)-\rho  \big(f'(u_n)u_n-f'(u(t_n))u(t_n)\big)\Big)}_{H^1}\\
& +\norm{\Phi_1(h \sqrt{\mathcal{A}}) \big(f( u_n)-f( u(t_n))\big)}_{H^1}\lesssim \norm{e_n^{u}}_{H^1},
\end{aligned}
\end{equation*}
 Based on the  Lipschitz continuity conditions, Kato–Ponce inequalities and \eqref{method-coefficient-bound},
it is derived that
 \begin{equation*}
\begin{aligned}
&\norm{h   \Psi_1(h \sqrt{\mathcal{A}})\big(f'( u_n)v_n-f'( u(t_n))v(t_n)\big)}_{H^1}\\
=&
\norm{h \sqrt{\mathcal{A}}  \Psi_1(h \sqrt{\mathcal{A}}) \frac{f'( u_n)v_n-f'( u(t_n))v(t_n)}{\sqrt{\mathcal{A}}} }_{H^1}
\\\lesssim&\norm{ f'( u_n)v_n-f'( u(t_n))v(t_n) }_{L^2}\\
\lesssim&\norm{f'( u_n)}_{H^{\mu}}\norm{e_n^{v}}_{L^2}+
\norm{f'( u_n) -f'( u(t_n)) }_{L^2}\norm{v(t_n)}_{H^{\mu}}
\\\lesssim& \norm{e_n^{v}}_{L^2}+ \norm{e_n^{u}}_{L^2},
\end{aligned}
\end{equation*}
and
 \begin{equation*}
\begin{aligned}
&\norm{h^2  \Psi_2(h \sqrt{\mathcal{A}})  \big(f''(u_n)v_n^2-f''(u(t_n))v(t_n)^2\big)}_{H^1}\\
=&h
\norm{h \sqrt{\mathcal{A}}  \Psi_2(h \sqrt{\mathcal{A}}) \frac{f''(u_n)v_n^2-f''(u(t_n))v(t_n)^2}{\sqrt{\mathcal{A}}} }_{H^1}\\
\lesssim&h\norm{f''(u_n)v_n^2-f''(u(t_n))v(t_n)^2}_{L^2}
\lesssim h(\norm{e_n^{v}}_{L^2}+ \norm{e_n^{u}}_{L^2}).
\end{aligned}
\end{equation*}
With the same arguments, we get
 \begin{equation*}
\begin{aligned}
&\norm{h^2  \Psi_2(h \sqrt{\mathcal{A}})  \big(f''(u_n)(\nabla u_n)^2-f''(u(t_n))(\nabla u(t_n))^2\big)}_{H^1}\\
=&
\norm{h^2  \mathcal{A}   \Psi_2(h \sqrt{\mathcal{A}}) \frac{f''(u_n)(\nabla u_n)^2-f''(u(t_n))(\nabla u(t_n))^2}{ \mathcal{A}} }_{H^1}\\
\lesssim  &
\norm{ \frac{f''(u_n)(\nabla u_n)^2-f''(u(t_n))(\nabla u(t_n))^2}{ \mathcal{A}} }_{H^1}\\
\lesssim  &
\norm{ \frac{\big(f''(u_n)-f''(u(t_n))\big)(\nabla u_n)^2}{ \sqrt{\mathcal{A}}} }_{L^2} +\norm{ \frac{f''(u(t_n))\big((\nabla u_n)^2-(\nabla u(t_n))^2\big)}{\sqrt{ \mathcal{A}}} }_{L^2}.
\end{aligned}
\end{equation*}
According to the results of Lemma \ref{lamma0} and the boundedness \eqref{bound num}, we find
 \begin{equation*}
\begin{aligned}
&
\norm{ \frac{\big(f''(u_n)-f''(u(t_n))\big)(\nabla u_n)^2}{ \sqrt{\mathcal{A}}} }_{L^2}
\lesssim   \norm{f''(u_n)-f''(u(t_n))}_{L^2}\norm{u_n\nabla u_n}_{H^{\mu}} \\
\lesssim &
\norm{u^2_n}_{H^{1+\mu}}\norm{f''(u_n)-f''(u(t_n))}_{L^2} \lesssim  \norm{e_n^{u}}_{H^1},
\end{aligned}
\end{equation*}
and
 \begin{equation*}
\begin{aligned}
&
\norm{ \frac{f''(u(t_n))\big((\nabla u_n)^2-(\nabla u(t_n))^2\big)}{\sqrt{ \mathcal{A}}} }_{L^2}
\lesssim   \norm{f''(u(t_n)}_{H^{\mu}}\norm{\nabla u^2_n - \nabla u^2(t_n) }_{L^2}\\
\lesssim &
\norm{  u^2_n - u^2(t_n) }_{H^1}\norm{f''(u(t_n)}_{H^{\mu}}
\lesssim  \norm{  u_n + u(t_n) }_{H^{\max(\mu,1)}} \norm{e_n^{u}}_{H^1} \lesssim \norm{e_n^{u}}_{H^1}.
\end{aligned}
\end{equation*}
Here the boundedness \eqref{bound num} of numerical solution is used in the derivation.
Therefore, it is obtained that
\begin{equation*}
\begin{aligned}
&\norm{h^2  \Psi_2(h \sqrt{\mathcal{A}})  \big(f''(u_n)(\nabla u_n)^2-f''(u(t_n))(\nabla u(t_n))^2\big)}_{H^1} \lesssim  \norm{e_n^{u}}_{H^1}.
\end{aligned}
\end{equation*}

Combining these results with the first formula of \eqref{sta errors} gives
$$
\norm{\eta_n^{u}}_{H^1}
\lesssim  \norm{e_n^{u}}_{H^1}+\norm{e_n^{v}}_{L^2}.
$$
In a similar way,  we can get the same result for $\eta_n^{v}$:$$
\norm{\eta_n^{v}}_{L^2}
\lesssim  \norm{e_n^{u}}_{H^1}+\norm{e_n^{v}}_{L^2}.$$

$\bullet$
 Now we turn back to the recurrence relation \eqref{equ errors} which is rewritten in the form:
\begin{equation*}
\begin{aligned}
&\left(
  \begin{array}{c}
  e_{n+1}^{u}\\
   e_{n+1}^{v} \\
  \end{array}
\right)=\left(
  \begin{array}{cc}
         \cos(h \sqrt{\mathcal{A}})       & h  \sinc( h \sqrt{\mathcal{A}}) \\
    - h \mathcal{A} \sinc( h \sqrt{\mathcal{A}})  &  \cos(h \sqrt{\mathcal{A}}) \\
  \end{array}
\right) \left(
  \begin{array}{c}
    e_{n}^u \\
     e_{n}^v \\
  \end{array}
\right)
+ \left(
          \begin{array}{c}
            \zeta_n^{u}+h^2\eta_n^{u}\\
            \zeta_n^{v}+h\eta_n^{v}\\
          \end{array}
        \right).
\end{aligned}
\end{equation*}
It is bounded by
\begin{equation*}
\begin{aligned}
&\norm{\left(
  \begin{array}{c}
  e_{n+1}^{u}\\
   e_{n+1}^{v} \\
  \end{array}
\right)}_1\leq \norm{\left(
  \begin{array}{c}
    e_{n}^u \\
     e_{n}^v \\
  \end{array}
\right)}_1
+ \norm{\left(
          \begin{array}{c}
            \zeta_n^{u}+h^2\eta_n^{u}\\
            \zeta_n^{v}+h\eta_n^{v}\\
          \end{array}
        \right)}_1\\
        \leq& \norm{\left(
  \begin{array}{c}
    e_{n}^u \\
     e_{n}^v \\
  \end{array}
\right)}_1
+  \sqrt{(\norm{\zeta_n^{u}}_{H^1}+h^2\norm{\eta_n^{u}}_{H^1})^2+
   (\norm{\zeta_n^{v}}_{L^2}+h^2\norm{\eta_n^{v}}_{L^2})^2}\\
           \leq& \norm{\left(
  \begin{array}{c}
    e_{n}^u \\
     e_{n}^v \\
  \end{array}
\right)}_1
+  C\sqrt{h^8+h^6( \norm{e_n^{u}}_{H^1}+\norm{e_n^{v}}_{L^2})+h^4(\norm{e_n^{u}}^2_{H^1}+\norm{e_n^{v}}^2_{L^2})}\\
           \leq& \norm{\left(
  \begin{array}{c}
    e_{n}^u \\
     e_{n}^v \\
  \end{array}
\right)}_1
+  C h^4+Ch^3 \sqrt{\norm{e_n^{u}}_{H^1}+\norm{e_n^{v}}_{L^2}}+Ch^2 \sqrt{\norm{e_n^{u}}^2_{H^1}+\norm{e_n^{v}}^2_{L^2}}.
\end{aligned}
\end{equation*}
Applying the Gronwall inequality  yields
$
\sqrt{\norm{e_{n+1}^{u}}^2_{H^1}+\norm{e_{n+1}^{v}}^2_{L^2}}\leq C h^3.
$
This shows \eqref{error bound} exactly.  Therefore,  the theorem is
confirmed.
\end{proof}

\subsection{Convergence of fully-discrete scheme}

\begin{mytheo}  \label{fully-discrete thm}  Under the conditions of Theorem \ref{symplectic thm} and the  regularity condition $$(u(0,x), \partial_tu(0,x))\in [H^{1+\tilde{\mu}}(\mathbb{T}^d)\bigcap H^{\tilde{\mu}}_0(\mathbb{T}^d)]\times H^{\tilde{\mu}}(\mathbb{T}^d)$$
with $\tilde{\mu}=\max(\mu,1)$ and $\mu> \frac{d}{2}$,
the numerical solution produced by the fully-discrete scheme  \eqref{method-f} has the following error bound:
 \begin{equation}\label{error bound new}
\begin{aligned}
 \norm{\Pi_{N_x}U(t_{n+1})-U_{n+1}}_{H^{1}}+ \norm{\Pi_{N_x}V(t_{n+1})-V_{n+1}}_{L^{2}}  \leq C\big( h ^3+ N_x^{-1-\tilde{\mu}}\big),
\end{aligned}
\end{equation}
where $0\leq n \leq \frac{T}{h}-1$, and $C$ is the error constant which only depends on $T$ and $C_0$ given in Theorem \ref{symplectic thm}.
\end{mytheo}
\begin{remark}
From this result,  it follows that by passing to the limit $N_x $,  the convergence of   semi-discretization given in Theorem
\ref{symplectic thm} is obtained. In practical applications, a large $N_x $ can be chosen and then the main error of fully-discrete scheme
comes from the time  discretization.
\end{remark}

\begin{proof}
Denote  the errors of the fully-discrete solution \eqref{method-f} by $$E_n^U=\Pi_{N_x}U(t_n)-U_n,\ \ \ E_n^V=\Pi_{N_x}V(t_n)-V_n.$$
By the construction of the  semi-discrete scheme presented in Section \ref{sds-construct}, it is obtained that
 the exact solution satisfies
 \begin{equation}\label{exact-ff}
\begin{aligned}
\Pi_{N_x}U(t_{n+1})=&\cos(h \sqrt{\mathcal{A}}) \Pi_{N_x}U(t_n)+h  \sinc( h \sqrt{\mathcal{A}})  \Pi_{N_x}V(t_n)\\
&+h^2\Phi_1(h \sqrt{\mathcal{A}})\Pi_{N_x} f( U(t_n))+h^3  \Psi_1(h \sqrt{\mathcal{A}})\Pi_{N_x} \big(f'( U(t_n))V_n\big)\\& +h^4  \Psi_2(h \sqrt{\mathcal{A}}) \Pi_{N_x} \Big(f''(U(t_n))V(t_n)^2-  f''(U(t_n))(\nabla U_n)^2\\&+\rho  f(U(t_n))-\rho  f'(U(t_n))U(t_n)\Big)\\
&+\Pi_{N_x} \Big(R_1(t_n)+\textmd{Part II}^{u}+\textmd{Part III}^{u}\Big),\\
\Pi_{N_x}V(t_{n+1})
   =&-h \mathcal{A} \sinc( h \sqrt{\mathcal{A}}) \Pi_{N_x}U(t_n)+\cos(h \sqrt{\mathcal{A}}) \Pi_{N_x}V(t_n)\\&+
h  \Phi_2(h \sqrt{\mathcal{A}}) \Pi_{N_x} f( U(t_n)) +h^2\Phi_1(h \sqrt{\mathcal{A}})\Pi_{N_x}\big(f'(U(t_n))V(t_n)\big)\\& +h^3  \Psi_1(h \sqrt{\mathcal{A}})  \Pi_{N_x}\Big(f''(U(t_n))(V(t_n)^2-  (\nabla U(t_n))^2)\\&+\rho  f(U(t_n))- f'(U(t_n))(\rho U(t_n)- f(U(t_n)))\Big)\\&+\Pi_{N_x} \Big(R_2(t_n)+R_3(t_n)+\textmd{Part III}^{v}\Big).
\end{aligned}
\end{equation}
Then considering the difference between  \eqref{method-f} and \eqref{exact-ff},  the following error equation is obtained:
 \begin{equation}\label{d-ff1}
\begin{aligned}
E_{n+1}^U =& \cos(h \sqrt{\mathcal{A}})E_n^U+h  \sinc( h \sqrt{\mathcal{A}}) E_n^V\\
&+h^2\Phi_1(h \sqrt{\mathcal{A}})\Pi_{N_x} \big(f( U(t_n))-f( U_n)\big)\\
&+h^3  \Psi_1(h \sqrt{\mathcal{A}})\Pi_{N_x} \big(f'(U(t_n))V(t_n)-f'( U_n)V_n\big)\\
&+h^4  \Psi_2(h \sqrt{\mathcal{A}}) \Pi_{N_x}
 \Big(f''(U(t_n))V(t_n)^2-f''(U_n)V_n^2\\
&+f''(U_n)(\nabla U_n)^2-  f''(U(t_n))
 (\nabla U_n)^2+\rho  f(U(t_n))\\
&-\rho  f(U_n)+\rho  f'(U_n)U_n-\rho  f'(U(t_n))U(t_n)\Big)\\
&+\Pi_{N_x} \Big(R_1(t_n)+\textmd{Part II}^{u}+\textmd{Part III}^{u}\Big)\\
&+\widetilde{R^U_1}(t_n)+\widetilde{R^U_2}(t_n)+\widetilde{R^U_3}(t_n),
\end{aligned}
\end{equation}
and
 \begin{equation}\label{d-ff2}
\begin{aligned}
E_{n+1}^V
   =&-h \mathcal{A} \sinc( h \sqrt{\mathcal{A}}) E_n^U+\cos(h \sqrt{\mathcal{A}}) E_n^V\\
&+
h  \Phi_2(h \sqrt{\mathcal{A}})  \quad \Pi_{N_x} \big(f( U(t_n))-f( U_n)\big)\\
&+h^2\Phi_1(h \sqrt{\mathcal{A}}) \Pi_{N_x}\big(f'(U(t_n))V(t_n)-f'( U_n)V_n\big)  \\&+h^3  \Psi_1(h \sqrt{\mathcal{A}})  \Pi_{N_x}
\Big(f''(U(t_n))V(t_n)^2-f''(U_n)V_n^2\\
&+f''(U_n)(\nabla u_n)^2 -  f''(U(t_n))(\nabla U_n)^2+\rho  f(U(t_n))\\
&-\rho  f(U_n)+\rho  f'(U_n)U_n-\rho  f'(U(t_n))U(t_n)\\
&+f'( U(t_n))f( U(t_n))-f'(U_n)f(U_n)\Big)\\
&+\Pi_{N_x} \Big(R_2(t_n)+R_3(t_n)+\textmd{Part III}^{v}\Big)\\&+\widetilde{R^V_1}(t_n)+\widetilde{R^V_2}(t_n)+\widetilde{R^V_3}(t_n),
\end{aligned}
\end{equation}
where we introduce the following notations to denote remainders
 \begin{equation}\label{nd-ff}
\begin{aligned}
 \widetilde{R^U_1}(t_n) =&h^2\Phi_1(h \sqrt{\mathcal{A}})(\Pi_{N_x}-I_{N_x}) f( U_n),\\
  \widetilde{R^U_2}(t_n) =&h^3  \Psi_1(h \sqrt{\mathcal{A}})(\Pi_{N_x}-I_{N_x})\big(f'( U_n)V_n\big),\\
   \widetilde{R^U_3}(t_n) =&h^4  \Psi_2(h \sqrt{\mathcal{A}})(\Pi_{N_x}-I_{N_x})\Big(f''(U_n)V_n^2-  f''(U_n)(\nabla U_n)^2\\&\qquad\qquad\qquad\qquad\qquad\quad+\rho  f(U_n)-\rho  f'(U_n)U_n\Big),
\end{aligned}
\end{equation}
and
 \begin{equation}\label{nd2-ff}
\begin{aligned}
 \widetilde{R^V_1}(t_n) =&h  \Phi_2(h \sqrt{\mathcal{A}})(\Pi_{N_x}-I_{N_x}) f( U_n),\\
  \widetilde{R^V_2}(t_n) =&h^2\Phi_1(h \sqrt{\mathcal{A}})(\Pi_{N_x}-I_{N_x})\big(f'( U_n)V_n\big),\\
   \widetilde{R^V_3}(t_n) =&h^3  \Psi_1(h \sqrt{\mathcal{A}})(\Pi_{N_x}-I_{N_x}) \Big(f''(U_n)\big(V_n^2-  (\nabla U_n)^2\big) \\&\qquad\qquad\qquad\qquad\qquad\quad+\rho  f(U_n)-  f'(U_n)\big(\rho U_n-f(U_n)\big)\Big).
\end{aligned}
\end{equation}
With the same analysis as  the bounds \eqref{ II bound}, \eqref{ III bound}-\eqref{ IIIv bound} and Lemmas \ref{lamma1}--\ref{lamma2}, we have
$$\begin{aligned}&\norm{\Pi_{N_x} \Big(R_2(t_n)+R_3(t_n)+\textmd{Part III}^{v}\Big)}_{L^2}\\&+\norm{\Pi_{N_x} \Big(R_1(t_n)+\textmd{Part II}^{u}+\textmd{Part III}^{u}\Big)}_{H^1} \lesssim h ^4.\end{aligned}$$
In what follows, we derive the bounds for the terms \eqref{nd-ff}-\eqref{nd2-ff}   by using mathematical induction on $n$:  assuming that
 \begin{equation}\label{ass-ff}
\begin{aligned}
\norm{U_n}_{H^{1+\tilde{\mu}}}\leq \norm{\Pi_{N_x} U(t_n)}_{H^{1+\tilde{\mu}}}+1,\ \ \ \norm{V_n}_{H^{ \alpha}}\leq \norm{\Pi_{N_x} V(t_n)}_{H^{\tilde{\mu}}}+1,
\end{aligned}
\end{equation}
we shall prove the following results:
 \begin{equation}\label{res-ff}
\begin{aligned}
\norm{U_{n+1}}_{H^{1+\tilde{\mu}}}\leq \norm{\Pi_{N_x} U(t_{n+1})}_{H^{1+\tilde{\mu}}}+1,\ \ \ \norm{V_{n+1}}_{H^{\tilde{\mu}}}\leq \norm{\Pi_{N_x} V(t_{n+1})}_{H^{\tilde{\mu}}}+1.
\end{aligned}
\end{equation}
For the first two terms of \eqref{nd-ff}, it follows  from  \cite{LI23} and \eqref{ass-ff}  that
 \begin{equation*}
\begin{aligned}
\norm{ \widetilde{R^U_1}(t_n)}_{H^1} \lesssim&h  \norm{\sin(h \sqrt{\mathcal{A}})(\Pi_{N_x}-I_{N_x}) f( U_n)}_{L^{2}}
\lesssim h  \norm{ (\Pi_{N_x}-I_{N_x}) f( U_n)}_{L^{2}}\\
 \lesssim&hN_x^{-2}  \norm{f'( U_n)\nabla^2   U_n +f''( U_n)\nabla  U_n \otimes \nabla  U_n}_{L^{2}}  \\
 \lesssim&hN_x^{-2} \big(\norm{U_n}_{H^{2}}+\norm{\nabla U_n}^2_{L^{4}} \big) \\
 \lesssim&hN_x^{-2} \big(\norm{U_n}_{H^{2}}+\norm{U_n}^2_{H^{1+\frac{d}{4}}} \big) \\
\lesssim &hN_x^{-1-\tilde{\mu}}  \norm{U_n}_{H^{1+\tilde{\mu}}}+hN_x^{-2-\tilde{\mu}+\frac{d}{4}}  \norm{U_n}^2_{H^{1+\tilde{\mu}}}
\\
\lesssim &hN_x^{-1-\tilde{\mu}} \big( \norm{U_n}_{H^{1+\tilde{\mu}}}+ \norm{U_n}^2_{H^{1+\tilde{\mu}}}\big),
\end{aligned}
\end{equation*}
and similarly
 \begin{equation*}
\begin{aligned}
 \norm{ \widetilde{R^U_2}(t_n)}_{H^1}   \lesssim& hN_x^{-1-\tilde{\mu}} \big(\norm{U_n}_{H^{1+\tilde{\mu}}}+\norm{V_n}_{H^{\tilde{\mu}}}+\norm{U_n}^2_{H^{1+\tilde{\mu}}}+\norm{V_n}^2_{H^{\tilde{\mu}}} \big).
\end{aligned}
\end{equation*}
For the third one in \eqref{nd-ff}, by noticing
$$\Psi_2(h \sqrt{\mathcal{A}})= \frac{1}{4h^2\mathcal{A}} \sinc^2( h \sqrt{\mathcal{A}}/2) -\frac{1}{2h^2\mathcal{A}}\sinc( h \sqrt{\mathcal{A}}),$$
we deduce that
 \begin{equation*}
\begin{aligned}
 &\norm{  \widetilde{R^U_3}(t_n)}_{H^1}\\ \lesssim&  h  \norm{    \frac{\Pi_{N_x}-I_N}{\sqrt{\mathcal{A}}^3} \Big(f''(U_n)V_n^2-  f''(U_n)(\nabla U_n)^2+\rho  f(U_n)-\rho  f'(U_n)U_n\Big)}_{H^1}\\
 \lesssim&  h N_x^{-2} \norm{\frac{1}{\sqrt{\mathcal{A}}^3} \Big(f''(U_n)V_n^2-  f''(U_n)(\nabla U_n)^2+\rho  f(U_n)-\rho  f'(U_n)U_n\Big)}_{H^3}\\
  \lesssim&  h N_x^{-2} \norm{ f''(U_n)V_n^2-  f''(U_n)(\nabla U_n)^2+\rho  f(U_n)-\rho  f'(U_n)U_n }_{L^2}\\
  \lesssim&  h N_x^{-2} \big(\norm{ V_n}^2_{L^2}+ \norm{U_n}^2_{H^2}+\norm{ f(U_n)}_{L^2}+\norm{U_n }_{L^2}\big)\\
    \lesssim&  h N_x^{-1-\tilde{\mu}} \big(\norm{ V_n}^2_{H^{\tilde{\mu}}}+ \norm{U_n}^2_{H^{1+\tilde{\mu}}}+\norm{U_n}_{H^{1+\tilde{\mu}}} \big).
\end{aligned}
\end{equation*}
Similarly, we get
 \begin{equation*}
\begin{aligned}
\norm{ \widetilde{R^V_1}(t_n)}_{L^2} \lesssim&h N_x^{-1-\tilde{\mu}},\ \
 \norm{ \widetilde{R^V_2}(t_n)}_{L^2}   \lesssim h N_x^{-1-\tilde{\mu}},\ \
 \norm{ \widetilde{R^V_3}(t_n)}_{L^2}   \lesssim h N_x^{-1-\tilde{\mu}}.
\end{aligned}
\end{equation*}
By using these estimates and taking the energy norm   on both sides of \eqref{d-ff1} and \eqref{d-ff2}, we obtain \eqref{error bound new} and \eqref{res-ff}
 with the same arguments given in the proof of   Theorem \ref{symplectic thm}.  The proof is complete.
\end{proof}

\begin{remark}
 Theorem \ref{fully-discrete thm}  states that for the solution $u\in H^{1+\tilde{\mu}}(\mathbb{T}^d)$  the
the
error is $$  \norm{\Pi_{N_x}U(t_{n+1})-U_{n+1}}_{H^{1}}+ \norm{\Pi_{N_x}V(t_{n+1})-V_{n+1}}_{L^{2}}  \leq C\big( h ^3+N_x^{-1-\tilde{\mu}}\big).$$ This is   much better than the regularity of the solution in time and space.
In general,  the projection error in space  should be
$$ \norm{\Pi_{N_x}U(t_{n+1})-U(t_{n+1})}_{H^{1}}+ \norm{\Pi_{N_x}V(t_{n+1})-V(t_{n+1})}_{L^{2}}  \leq C N_x^{ -\tilde{\mu}}.$$
Similar improvement also  occurs for the second-order low-regularity integrator  researched  in   \cite{LI23}.

\end{remark}

\section{Numerical test}\label{sec:4}

In this section, we show the numerical performance of the proposed third-order low-regularity integrator (LRI) by
comparing it with the well-known exponential integrators (EIs) \cite{Hochbruck10}.
 We choose two third-order EIs  from \cite{Hochbruck06} (denoted by EI1 and EI2) and an  exponential fitting   TI  from \cite{wang12} (denoted by EI3). The reference solution is obtained by using a very fine stepsize for the third-order low-regularity integrator.

 We shall
present the results with an one-dimensional example of (\ref{klein-gordon})
for simplicity: the nonlinear  Klein--Gordon equation \eqref{klein-gordon} with $d=1, \rho=0,\mathbb{T}=(-\pi,\pi), f(u)=\sin(u)$
(we  note that  the scheme proposed in  \cite{Zhao23} is not applicable to this nonlinear function).
  We choose the initial values $\psi_1(x)$ and $\psi_2(x)$ in the same way  as described in Section 5.1 of \cite{O18} and Section 4 of \cite{Zhao23}. More precisely, the  initial values $\psi_1(x)$ and $\psi_2(x)$
are in the space  $H^{\theta}(\mathbb{T})\times H^{\theta-1}(\mathbb{T})$.

We first  verify the spatial accuracy of the
our method where  $h=10^{-6}$ is chosen such
 that the temporal error is
negligible compared to the spatial error.  Figure  \ref{p0new} plots the error:
$err=\frac{\norm{u_n-u(t_n)}_{H^{1}}}{\norm{u(t_n)}_{H^{1}}}+\frac{\norm{v_n-v(t_n)}_{L^{2}}}{\norm{v(t_n)}_{L^{2}}}$ under different $N_x$ and $\theta$, where the reference solution is obtained again by the LRI with $N_x = 2^{16}$.
From the results, it is seen that for the large  $N_x$  the
error  brought by the  spatial discretization   can be neglected. Thence, we choose $N_x=2^{10}$ in the experiment.

    \begin{figure}[t!]
\centering
\includegraphics[width=5.2cm,height=4.4cm]{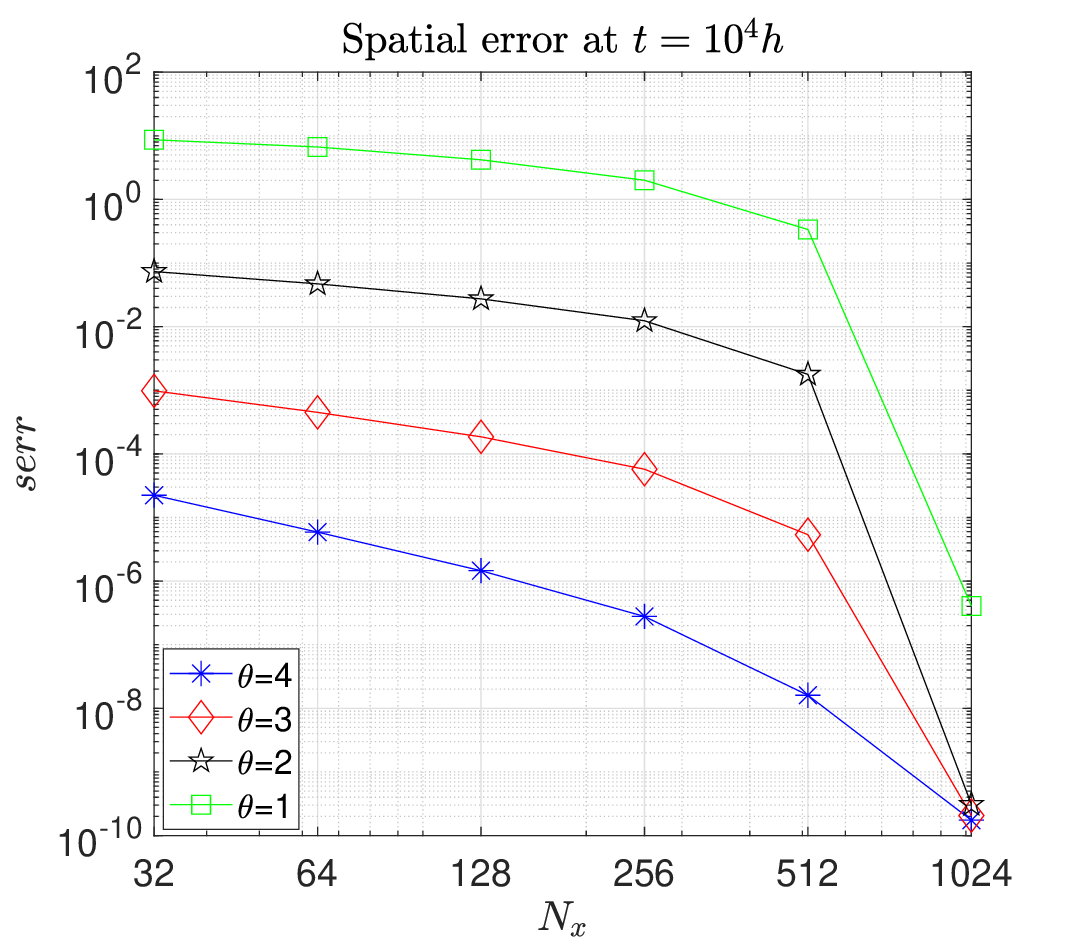}
\includegraphics[width=5.2cm,height=4.4cm]{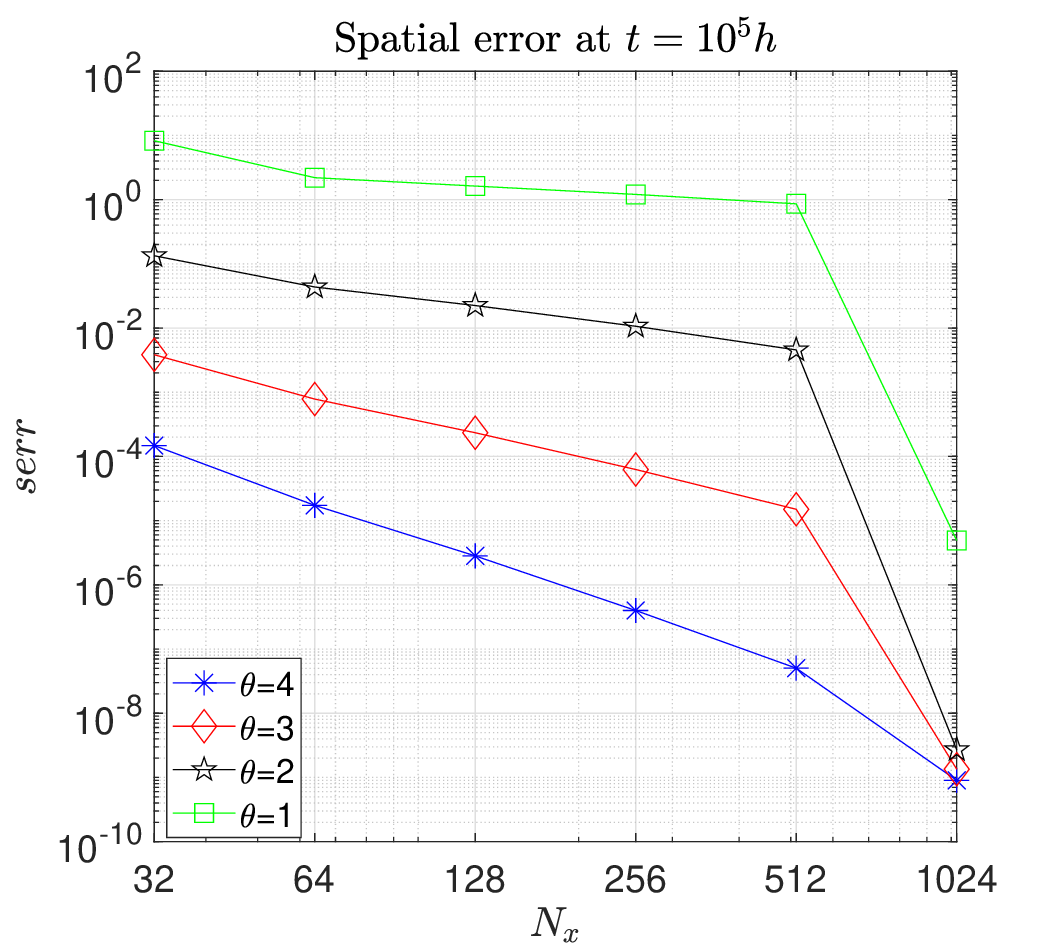}
\caption{Spatial error of LRI: the error $serr=\frac{\norm{u_n-u(t_n)}_{H^{1}}}{\norm{u(t_n)}_{H^{1}}}+\frac{\norm{v_n-v(t_n)}_{L^{2}}}{\norm{v(t_n)}_{L^{2}}}$ with  initial data $H^{\theta}(\mathbb{T}) \times H^{\theta-1}(\mathbb{T})$ against different $N_x$.} \label{p0new}
\end{figure}

In this test, we display  the global errors
 $err=\frac{\norm{u_n-u(t_n)}_{H^{1}}}{\norm{u(t_n)}_{H^{1}}}+\frac{\norm{v_n-v(t_n)}_{L^{2}}}{\norm{v(t_n)}_{L^{2}}}$ at $T=1$. The results for different initial value $(\psi_1,\psi_2)\in H^{\theta}(\mathbb{T})\times H^{\theta-1}(\mathbb{T})$  are given  in Figure \ref{p1}. From the results,  it can be observed that for the initial data in  $H^{\theta}(\mathbb{T})\times H^{\theta-1}(\mathbb{T})$  with $\theta=4, 3$, all the methods performance third order convergence. However, for the  initial data in a low-regularity space  $H^{\theta}(\mathbb{T})\times H^{\theta-1}(\mathbb{T})$  with $\theta=2$,  only the new integrator LRI proposed in this article shows the correct third-order convergence in $H^{1}(\mathbb{T})\times L^{2}(\mathbb{T})$. If the initial data is in a space $H^{\theta}(\mathbb{T})\times H^{\theta-1}(\mathbb{T})$  with $\theta=1.8,1.5, 1$ which is lower that the requirement $H^{2}(\mathbb{T})\times H^{1}(\mathbb{T})$ given in Theorem   \ref{symplectic thm},
 our method still has third-order accuracy and the expected bad behaviour does not show up for LRI in  Figure  \ref{p1}. To clarify this issue, we further display the result with more timesteps in  Figure \ref{fig1new}. The result  indicates that if the problem has lower regularity than $H^{2}(\mathbb{T})\times H^{1}(\mathbb{T})$, the new integrator LRI also does not show the correct order accuracy
 for very small time stepsize but it still performs much better than the exponential integrators.

 To investigate the practical gain from our proposed integrator, we study the efficiency of all the methods.
 Figure \ref{fig2}  displays the error at $T=5$ against the CPU time. It can be seen from the results that our proposed method LRI  can reach the same error level with remarkably less CPU time, and this clearly demonstrates the more efficiency of LRI  than the classical  exponential integrators.

  For the numerical experiments in two-dimensional  or three-dimensional space, the method has similar performance and we skip it for brevity.

    \begin{figure}[t!]
\centering
\includegraphics[width=3.9cm,height=4.4cm]{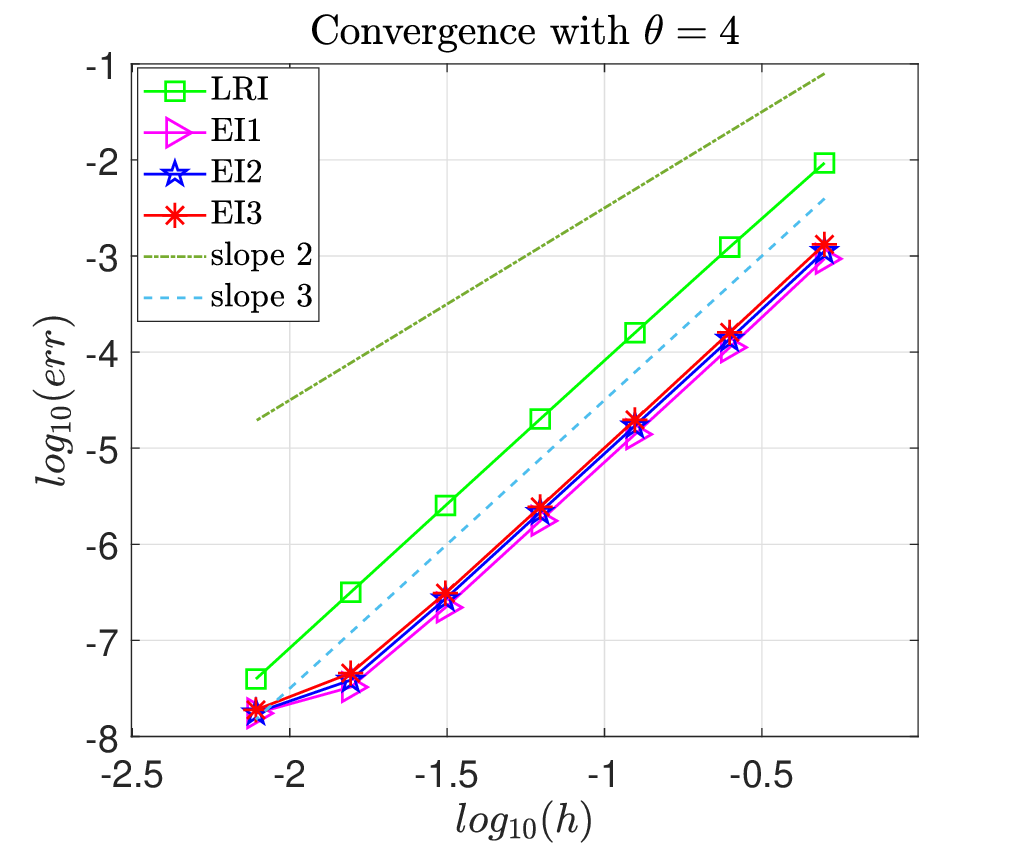}
\includegraphics[width=3.9cm,height=4.4cm]{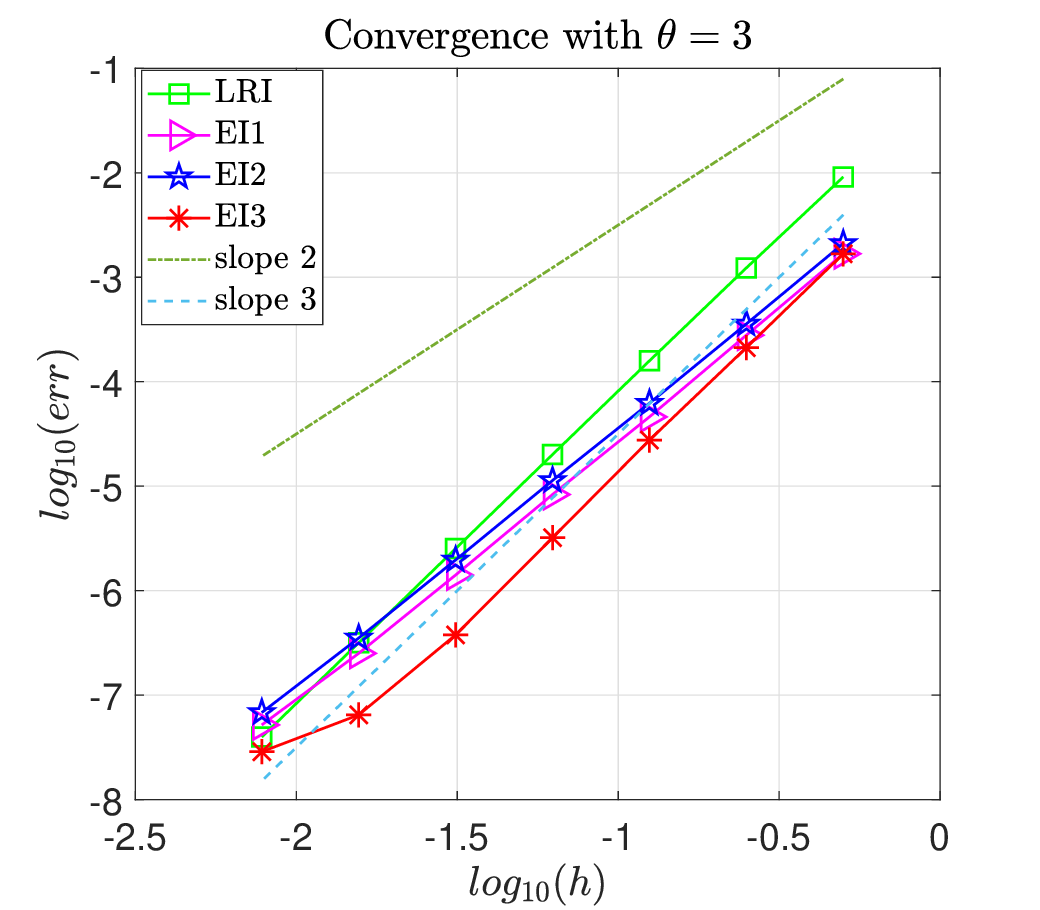}
\includegraphics[width=3.9cm,height=4.4cm]{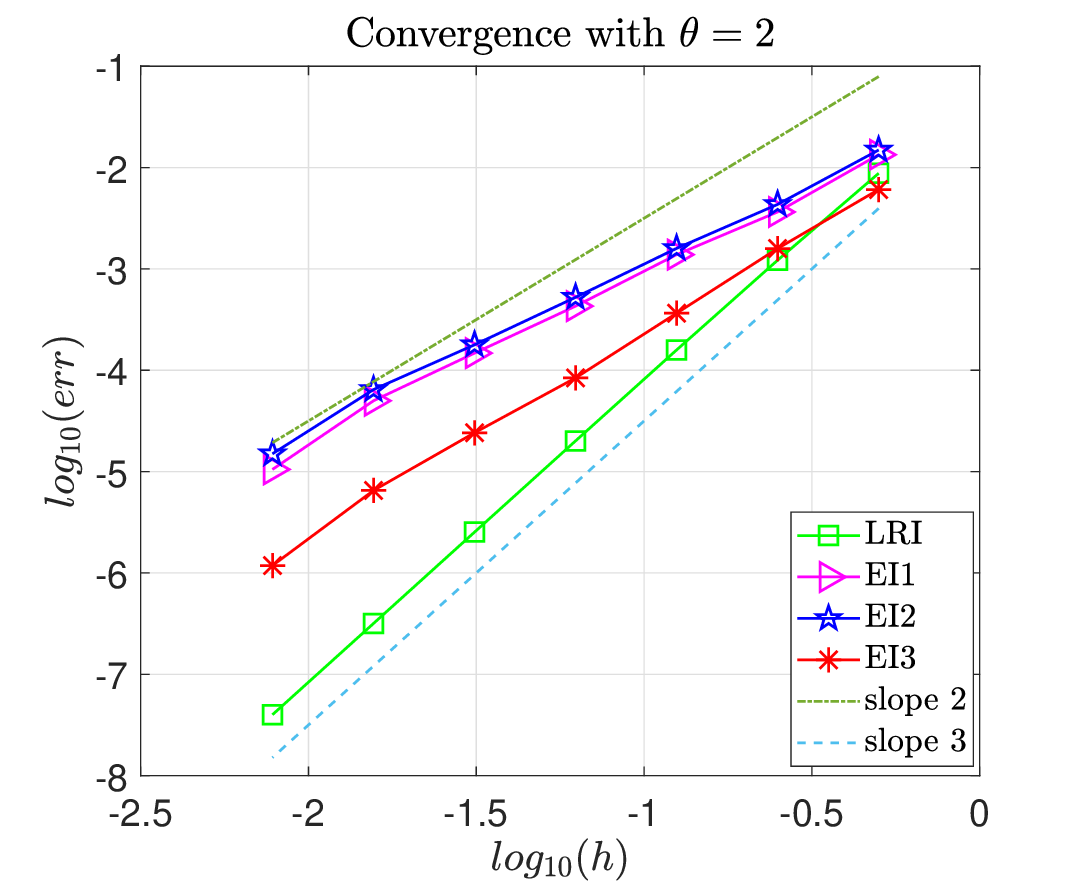}\\
\includegraphics[width=3.9cm,height=4.4cm]{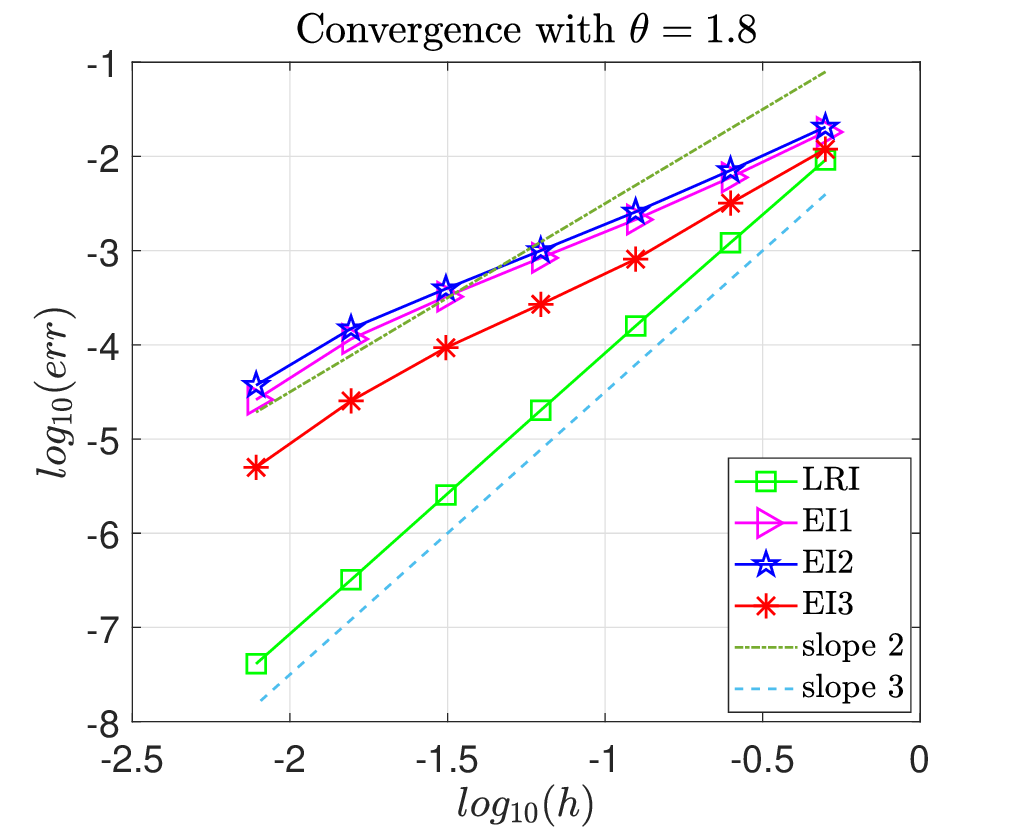}
\includegraphics[width=3.9cm,height=4.4cm]{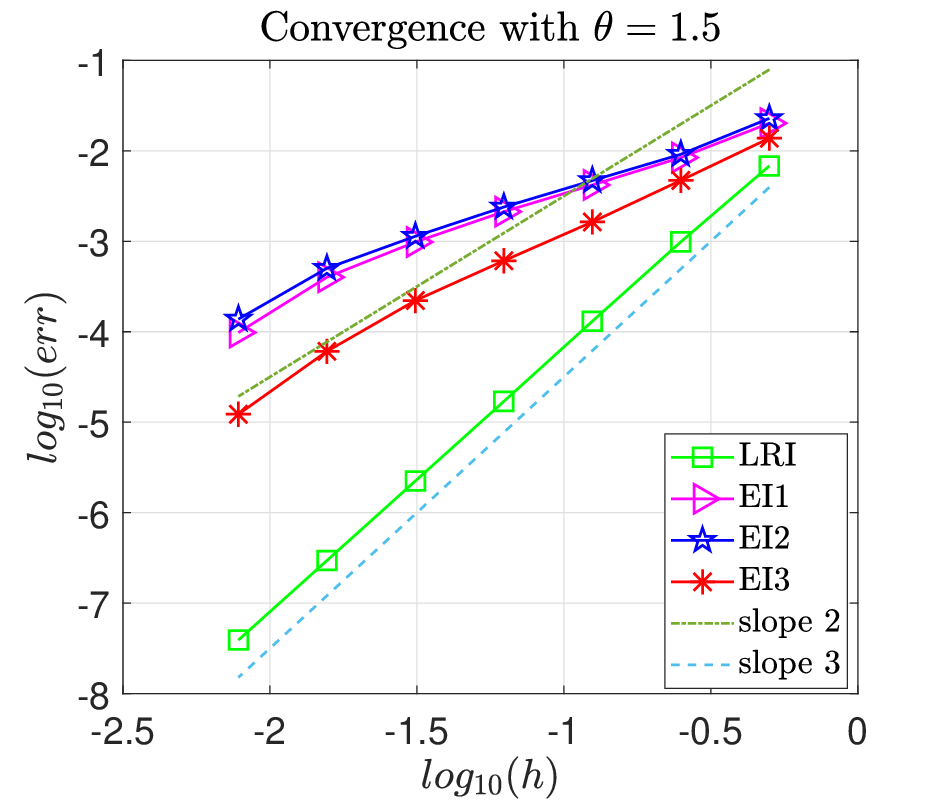}
\includegraphics[width=3.9cm,height=4.4cm]{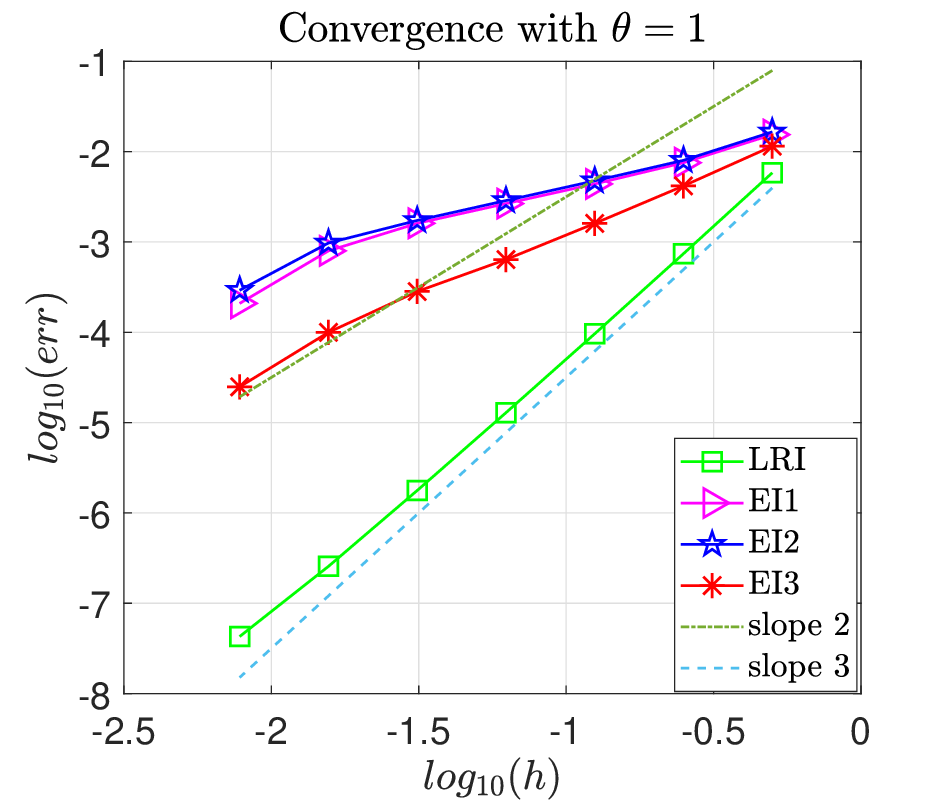}
\caption{Temporal error $err=\frac{\norm{u_n-u(t_n)}_{H^{1}}}{\norm{u(t_n)}_{H^{1}}}+\frac{\norm{v_n-v(t_n)}_{L^{2}}}{\norm{v(t_n)}_{L^{2}}}$ at $T=1$ with  initial data $H^{\theta}(\mathbb{T}) \times H^{\theta-1}(\mathbb{T})$ against $h $  with $h =1/2^k$, where $k=1,2,\ldots,7$.} \label{p1}
\end{figure}

    \begin{figure}[t!]
\centering
\includegraphics[width=3.9cm,height=4.4cm]{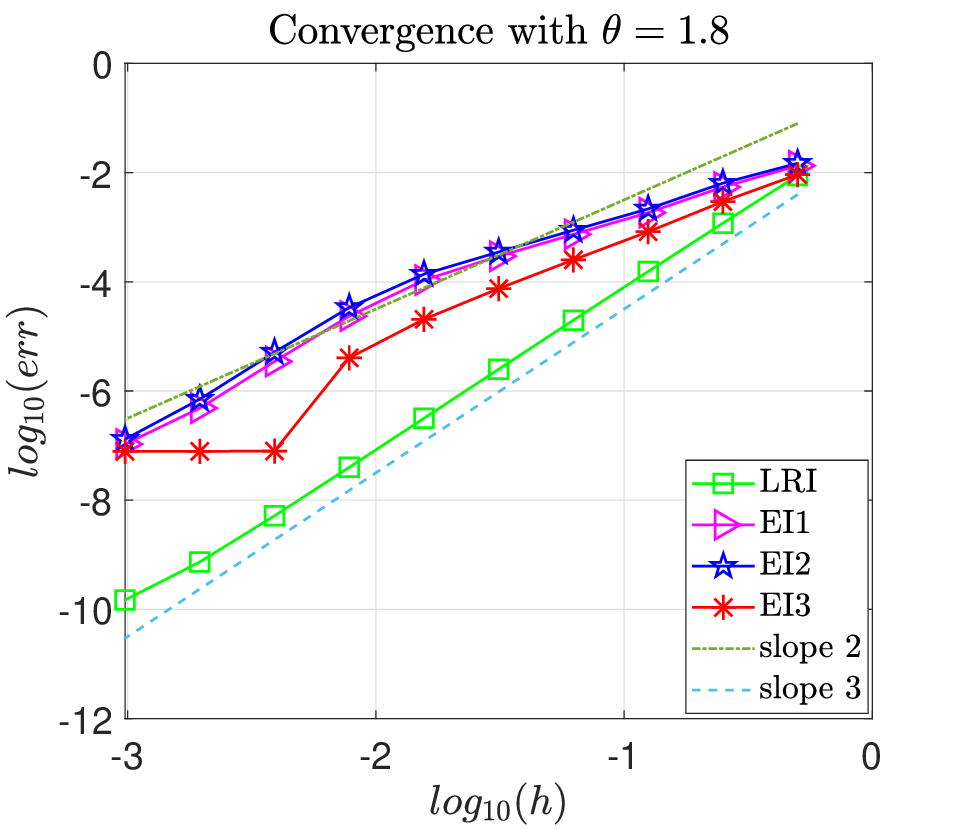}
\includegraphics[width=3.9cm,height=4.4cm]{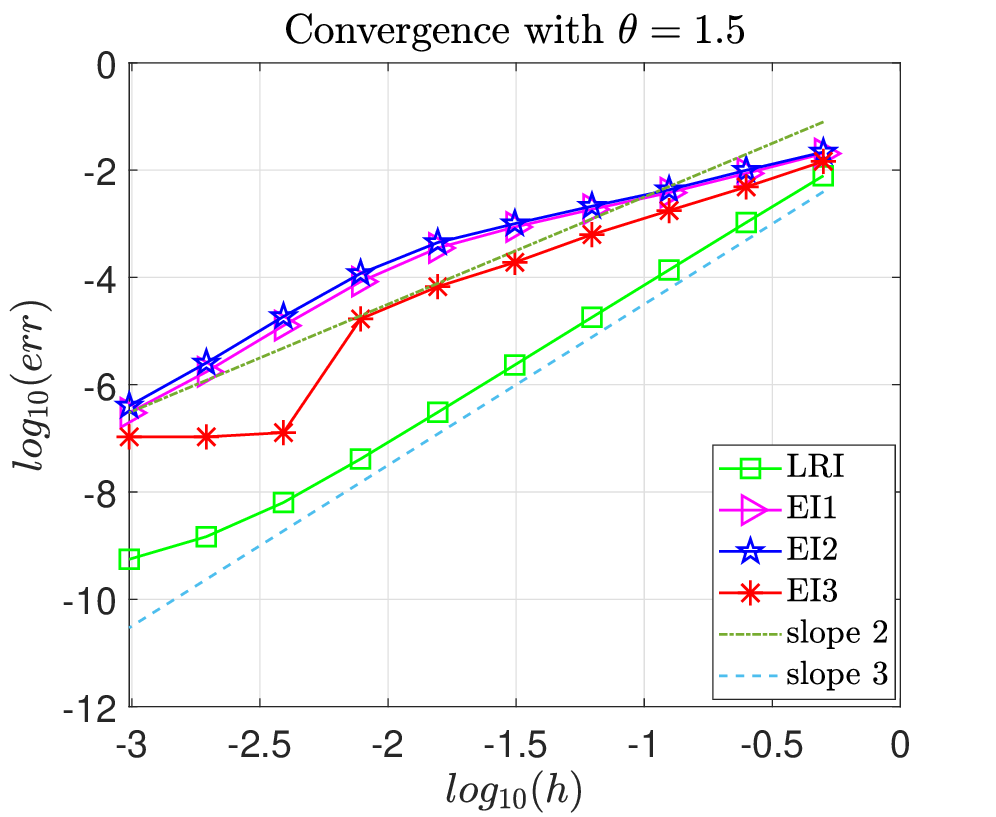}
\includegraphics[width=3.9cm,height=4.4cm]{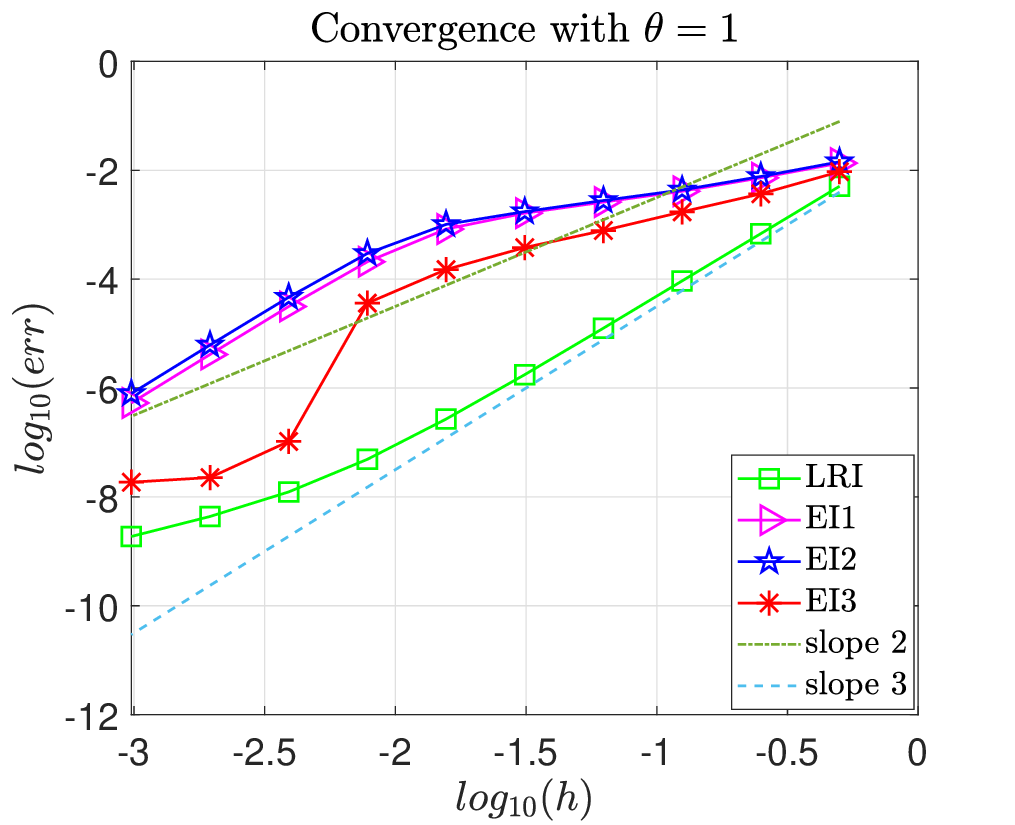}
\caption{Temporal error $err=\frac{\norm{u_n-u(t_n)}_{H^{1}}}{\norm{u(t_n)}_{H^{1}}}+\frac{\norm{v_n-v(t_n)}_{L^{2}}}{\norm{v(t_n)}_{L^{2}}}$ at $T=1$ with  initial data $H^{\theta}(\mathbb{T}) \times H^{\theta-1}(\mathbb{T})$ against $h $  with $h =1/2^k$, where $k=1,2,\ldots,10$.} \label{fig1new}
\end{figure}

    \begin{figure}[t!]
\centering
\includegraphics[width=3.9cm,height=4.4cm]{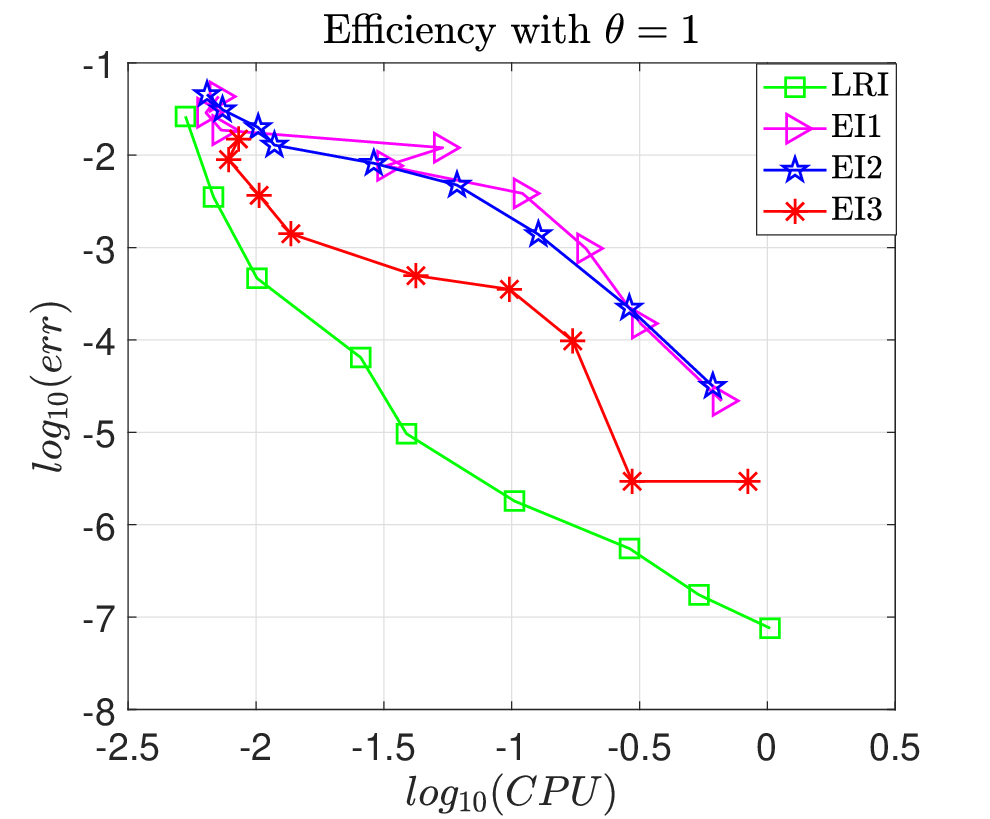}
\includegraphics[width=3.9cm,height=4.4cm]{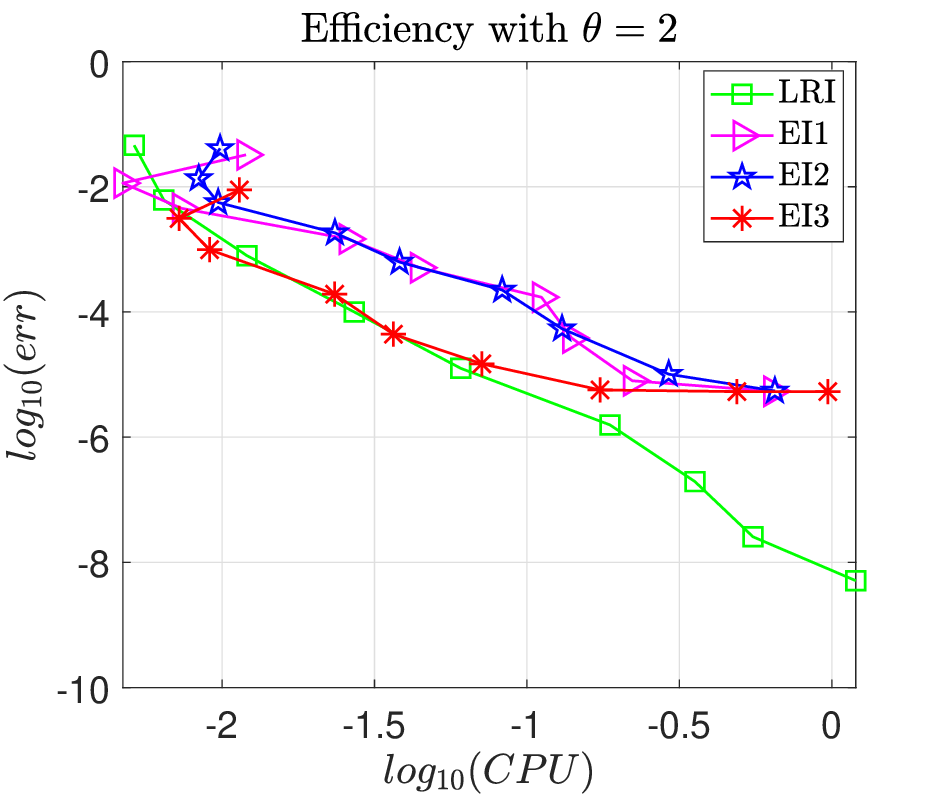}
\includegraphics[width=3.9cm,height=4.4cm]{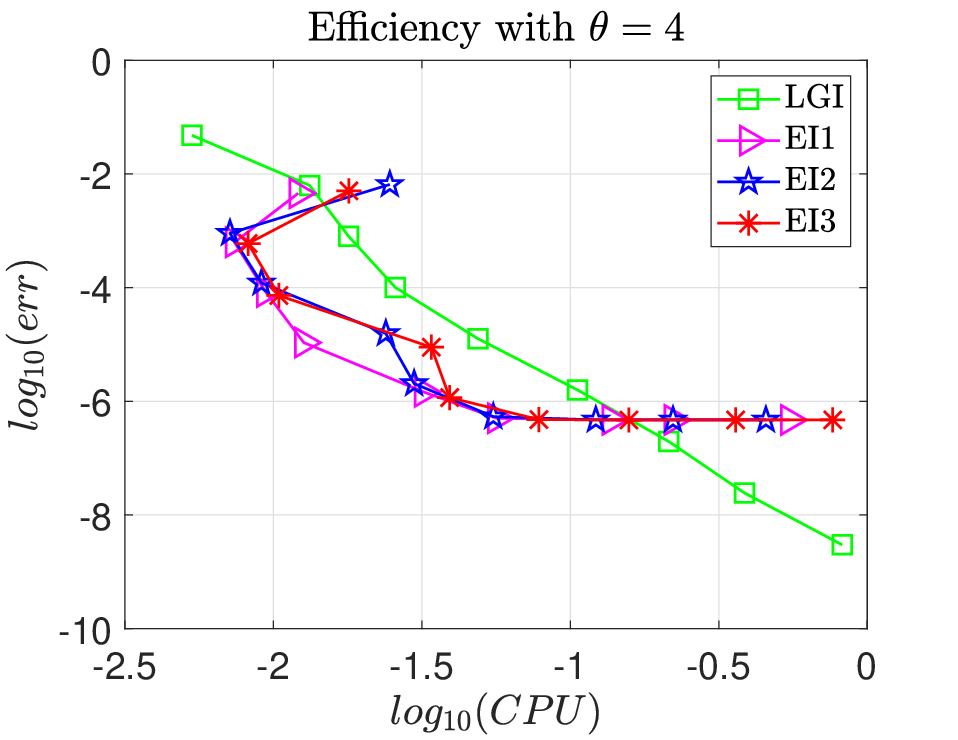}
\caption{Efficiency comparison: $err=\frac{\norm{u_n-u(t_n)}_{H^{1}}}{\norm{u(t_n)}_{H^{1}}}+\frac{\norm{v_n-v(t_n)}_{L^{2}}}{\norm{v(t_n)}_{L^{2}}}$ at $T=5$  against different CPU time produced by different  $h =1/2^k$, where $k=1,2,\ldots,7$.}\label{fig2}
\end{figure}

\section{Conclusions}\label{sec:5}
 In this paper,   a  low-regularity   trigonometric integrator was formulated and analysed  for
solving the nonlinear Klein-Gordon equation in the $d$-dimensional space with $d=1,2,3$.
Rigorous error estimates were given and the proposed integrator was shown to have
third-order time accuracy  in the energy space under a weak regularity condition.    A numerical experiment was carried out and the corresponding numerical results  were presented to  demonstrate  the behaviour of the new integrator  in comparison with   some well-known exponential integrators.

\section*{Acknowledgements}
This work was supported by NSFC (12371403, 12271426).


\end{document}